\documentclass[12pt, final]{amsart}
\usepackage[asymmetric, centering]{geometry}
\usepackage{graphicx}
\usepackage[colorlinks=true, citecolor=blue, linkcolor=blue, urlcolor=blue]{hyperref}
\usepackage{mathrsfs}
\usepackage{amsfonts}

\usepackage{enumerate}
\usepackage{amsmath}
\usepackage{amssymb}
\usepackage{fancyhdr}

\usepackage{bm}
\usepackage{ifdraft}
\ifoptionfinal{
\usepackage[disable]{todonotes}
}{
\usepackage[norefs, nocites]{refcheck}
\usepackage[breakall, fit]{truncate}
\usepackage{accsupp}
\setlength{\marginparsep}{5pt}

\usepackage{seqsplit}

\usepackage[notref, notcite]{showkeys}
\usepackage[bordercolor=white, color=white]{todonotes}
}
\newcommand{\HOX}[1]{\todo[noline, size=\footnotesize]{#1}}

\makeatletter\providecommand\@dotsep{5}\def\listtodoname{List of Todos}\def\listoftodos{\hypersetup{linkcolor=black}\@starttoc{tdo}\listtodoname\hypersetup{linkcolor=blue}}\makeatother

\newtheorem{lemma}{Lemma}
\newtheorem{proposition}{Proposition}
\newtheorem{theorem}{Theorem}
\newtheorem{corollary}{Corollary}
\newtheorem{definition}{Definition}
\theoremstyle{remark}

\newtheorem{remark}{Remark}

\def\V{\mbox{Var}}

\def\R\re
\def\V{\bf V}
\def\W{\mathcal{W}}

\def \re{{\mathbb R}}

\def \C{{\mathbb C}}

\def \J{{\mathbb J}}
\def \V{{\bf V}}
\def \M{{\mathcal M}}

\def \V{{\mathcal V}}

\def \g{{\mathfrak g}}

\newcommand{\id}{\mathrm{Id}}

\def\C{\mathbb C}
\def\R{\mathbb R}

\def\U{\mathbf U}
\def\P{\mathbf P}

\DeclareMathOperator{\GL}{GL}
\DeclareMathOperator{\Ad}{Ad}

\def \ba {\begin {eqnarray*} }
\def \ea {\end {eqnarray*} }
\def \beq {\begin {eqnarray}}
\def \eeq {\end {eqnarray}}

\def\g{\mathfrak g}

\def\p{\partial}
\DeclareMathOperator{\supp}{supp}

\DeclareMathOperator{\WF}{WF}

\title[Retrieving Yang--Mills--Higgs fields]{Retrieving Yang--Mills--Higgs fields in Minkowski space from active local measurements}

\author[X. Chen]{Xi Chen}
\address{Shanghai Center for Mathematical Sciences, Fudan University, Shanghai 200438, China;	Shanghai Academy of Artificial Intelligence for Science, Shanghai 200232, China;
Center for Applied Mathematics, Fudan University, Shanghai 200433, China. }
\email{xi\_chen@fudan.edu.cn}

\author[M. Lassas]{Matti Lassas}
\address{Department of Mathematics and Statistics, University of Helsinki, Helsinki FI-00014, Finland.  }
\email{Matti.Lassas@helsinki.fi}

\author[L. Oksanen]{Lauri Oksanen}
\address{Department of Mathematics and Statistics, University of Helsinki, Helsinki FI-00014, Finland.  }
\email{Lauri.oksanen@helsinki.fi}

\author[G.P. Paternain]{Gabriel P. Paternain}

\address{ Department of Pure Mathematics and Mathematical Statistics, University of Cambridge, Cambridge CB3 0WB, UK; Department of Mathematics, University
	of Washington, Seattle, WA 98195, USA.}
\email{g.p.paternain@dpmms.cam.ac.uk}

\begin{document}
\begin{abstract}  We show that we can retrieve a Yang--Mills potential and a Higgs field (up to gauge) from source-to-solution type data associated with the classical Yang--Mills--Higgs equations in Minkowski space $\mathbb{R}^{1+3}$. We impose natural non-degeneracy conditions on the representation for the Higgs field and on the Lie algebra of the structure group which are satisfied for the case of the Standard Model. Our approach exploits the non-linear interaction of waves generated by sources with values in the centre of the Lie algebra showing that abelian components can be used effectively to recover the Higgs field. 

\end{abstract}
\maketitle
\tableofcontents


\section{Introduction and overview} The goal of this paper is to solve an inverse problem for the classical Yang--Mills--Higgs equations in Minkowski space $\mathbb{R}^{1+3}$. The data for the problem arises from observations in a small set and the recovery is achieved on a causal domain where waves can propagate and return. The results here build upon our earlier work in \cite{CLOP,CLOP2} where we considered a cubic wave equation for the connection wave operator and the case of pure Yang--Mills theories, but this time we have the distinct objective of retrieving the Higgs field in addition to the connection (Yang--Mills field). The new couplings being introduced make the analysis of the non-linear interaction of waves quite complicated at first glance. The new insight is that this complication may be considerably avoided by focusing on sources taking values in the centre of the Lie algebra.

\subsection{The Yang--Mills--Higgs equations} We begin this overview by describing our setting in detail.
 Consider the trivial vector bundles $\Ad = M\times \g$
 and $E=M\times \mathcal{W}$, where \begin{itemize}
 	\item $(M,g)$ is the Minkowski space $\mathbb{R}^{1 + 3}$ with signature $(-, +, +, +)$;
 	\item $G$ is a compact connected matrix Lie group;
 	\item $\g$ is the Lie algebra of $G$ with an $\text{Ad}$-invariant inner product $\langle \cdot, \cdot \rangle_{\text{Ad}}$;
 	\item $\rho:G\to \text{GL}(\mathcal{W})$ is a complex linear representation of $G$ in the complex vector space       $\mathcal{W}$,
\item $\langle \cdot, \cdot \rangle_\mathcal W$ is a Hermitian $G$-invariant inner product on $\mathcal{W}$.
\end{itemize}
 

 We denote by $\Omega^k(M, \g)$ and $\Omega^k(M, \mathcal{W})$ the $\g$-valued and $\mathcal{W}$-valued smooth $k$-forms. In particular,  $\Omega^1(M, \g) = C^\infty(M, T^\ast M \otimes \g)$.
  Given $A\in \Omega^{1}(M,\g)$, it induces  exterior covariant derivatives in $\Ad$ and $E$ respectively,  
 \begin{eqnarray*}
	D_A : \Omega^k(M, \g) &\longrightarrow& \Omega^{k+1}(M, \g) \\  W &\longmapsto& dW + [A,   W],
	\\
	d_A : \Omega^k(M, \W) &\longrightarrow& \Omega^{k+1}(M, \W) \\  \Upsilon &\longmapsto& d\Upsilon + \rho_\ast(A) \wedge \Upsilon.
\end{eqnarray*}
The covariant derivatives $D_A$ and $d_{A}$ are compatible with the inner products in $\g$ and $\mathcal{W}$ respectively.
 The formal adjoints of $D_A$ and $d_A$ are defined through the Hodge star operator $\star$ on $(M, g)$ to be \begin{eqnarray*}
 D_A^\ast &=&  \star D_A \star, \\	d_A^\ast &=& \star d_A \star.
 \end{eqnarray*}

The $\text{Ad}$-invariant inner product on $\g$ can be naturally combined with the wedge product to give a pairing
\[\dot{\wedge}:\Omega^{p}(M,\g)\times\Omega^{q}(M,\g)\to \Omega^{p+q}(M,\mathbb{R}),\]
where we use the dot to indicate that the inner product on $\g$ is being used. Similarly we can use the Hermitian inner product on $\mathcal W$, to obtain a pairing
\[\dot{\wedge}:\Omega^{p}(M,\mathcal W)\times\Omega^{q}(M,\mathcal W)\to \Omega^{p+q}(M,\mathbb{C}).\]
Using these pairings and the Hodge star operator of $M$ we may define $L^2$-inner products as
	\begin{align*}
(\alpha,\beta)_{L^{2},\text{Ad}}&:=\int _{M}\alpha\dot{\wedge}\star\beta, \qquad \text{for}\;\alpha,\beta\in \Omega^{k}(M,\g),
\\
(\theta,\varphi)_{L^{2},\mathcal{W}}&:=\int _{M}\theta\dot{\wedge}\star\varphi,\qquad \text{for}\;\theta,\varphi\in \Omega^{k}(M,\mathcal W).
	\end{align*}
It turns out that $D_{A}^*$ and $d_{A}^*$ are also formal adjoints of $D_{A}$ and $d_{A}$ respectively when using the $L^2$-inner products.

  The curvature form $F_A$ is the $\g$-valued $2$-form defined  by $$F_A = dA +  \frac{1}{2} [A, A].$$
The Yang--Mills gauge field modelled by $A \in \Omega^{1}(M,\g)$ describes the interactions of gauge bosons, while the Higgs field represented by $\Phi \in C^{\infty}(M,\mathcal{W})$ addresses the interactions of Higgs bosons.  The Yang--Mills--Higgs Lagrangian density (or 4-form), \begin{equation}\label{eqn : YMH Lagrangian}2\mathcal L_{\text{YMH}}:=F_{A}\dot{\wedge}\star F_{A}+ d_{A}\Phi \dot{\wedge}\star d_{A}\Phi+\mathcal{V}(|\Phi|^{2})\,d\text{vol},\end{equation} 
encodes and couples the action of these fields, where $|\Phi|$ is the shorthand for $ \langle \Phi, \Phi \rangle_{\mathcal{W}}^{1/2}$ and the potential $\mathcal{V}(|\Phi|^{2})$ is a Mexican hat potential, say $\mathcal{V}(s) = s^{2}/2-s$ for simplicity. 
The specific form of $\mathcal{V}$ does not play a role in the proofs. 
By the  principle of least action, a Yang--Mills--Higgs field 
	\begin{align*}
(A, \Phi) \in \Omega^{1}(M,\g) \times C^{\infty}(M,\mathcal{W})
	\end{align*}
 is a solution to the Euler-Lagrange equations of $\mathcal L_{\text{YMH}}$, which are known as the Yang--Mills--Higgs equations. It can be easily computed that they form the coupled system
 \begin{align}
  	&D_{A}^*F_{A}+\J_{\rho}(d_{A}\Phi,\Phi)=0,\label{eq:ymh1}\\
  	&d_{A}^*d_{A}\Phi +\mathcal{V}'(|\Phi|^{2})\Phi=0.\label{eq:ymh2}
  \end{align} 
 The product $\J_{\rho}(\cdot,\cdot)$ is defined as follows.

\begin{definition} The {\bf real} bilinear form $\J_{\rho}:\mathcal W\times \mathcal W\to \g$ is uniquely defined by
the equation
\begin{equation*}
\Re \langle v,\rho_{*}(X)w\rangle_{\mathcal W}=\langle\J_{\rho}(v,w),X\rangle_{\text{\rm Ad}},\;\;\text{\rm for\;all}\;X\in \g\;\text{\rm and}\;v,w\in\mathcal W.
\label{eq:ex2}
\end{equation*}
\end{definition}

\begin{remark}\label{rem_antisymmetry}{\rm Since the Hermitian inner product in $\mathcal W$ is invariant under $\rho(g)$ for all $g\in G$ we see that
$\rho_{*}(X)$ is skew-Hermitian for all $X\in \g$ and thus $\J_{\rho}$ is anti-symmetric. }
\end{remark}

\begin{definition} We shall say that the representation $\rho$ is \text{\em fully charged} if its derivative
$\rho_{*}:\g\to\text{\rm End}(\mathcal{W})$ satisfies 
 \begin{eqnarray}\label{def : ideally charged}
 	\mbox{ $\rho_{*}(X)w=0$ \text{\rm for all} $X\in \g$} &\Longrightarrow&w=0.\end{eqnarray}
\end{definition}
The condition of $\rho$ being fully charged is equivalent to $\J_{\rho}$ being a non-degenerate bilinear form.

\begin{remark} In the literature one often encounters different choices of signs and factors of $1/2$ in the definition of $\mathcal L_{\text{YMH}}$. These choices will not really affect the analysis that we will perform on the equations \eqref{eq:ymh1}--\eqref{eq:ymh2}.

\end{remark}


    A section $\mathbf{U}\in C^{\infty}(M,G)$ acts on a Yang--Mills--Higgs field $(A, \Phi)$ as follows
 \begin{equation}\label{eqn : gauge transform}(A,\Phi)\mapsto 
(A, \Phi) \cdot \U = (A \cdot \mathbf{U},\rho(\mathbf{U}^{-1})\Phi),
\end{equation}
 where   $A \cdot \mathbf{U}=\mathbf{U}^{-1}d\mathbf{U}+\mathbf{U}^{-1}A\mathbf{U}$.  Two Yang--Mills--Higgs fields $(A, \Phi)$ and $(B, \Xi)$ are said to be gauge equivalent if there exists a section $\mathbf{U}\in C^{\infty}(M,G)$ such that  \[(B, \Xi) = (A ,\Phi) \cdot \mathbf{U}.\] The $\Ad$-invariance of $(\g, \langle \cdot, \cdot \rangle_{\Ad})$ along with the $G$-invariance of $(\mathcal{W}, \langle \cdot, \cdot \rangle_\mathcal W)$ guarantees that the Yang--Mills--Higgs system \eqref{eq:ymh1}--\eqref{eq:ymh2} is gauge invariant; namely, if $(A, \Phi)$ is a Yang--Mills--Higgs field then $(A,\Phi)\cdot \mathbf{U}$ is also a Yang--Mills--Higgs field for any section $\mathbf{U}\in C^{\infty}(M,G)$.

 \subsection{The inverse problem}

Following \cite{CLOP, CLOP2}, we shall work in the causal diamond
 $$
 \mathbb D
 := \{ (t,x) \in \R^{1+3} : |x| \le t + 1,\ |x| \le 1 - t \},
 $$
see Figure~\ref{fig_D} below for a visualization of $\mathbb D$.
Let $(A, \Phi)  \in \Omega^{1}(\mathbb{D},\g) \times C^\infty(\mathbb{D}, \mathcal{W})$ be a Yang--Mills--Higgs field i.e. $(A, \Phi)$ solves the Yang--Mills--Higgs system \eqref{eq:ymh1}--\eqref{eq:ymh2} in $\mathbb{D}$. The inverse problem we shall solve addresses whether $(A, \Phi)$ can be uniquely determined by some local measurements in $\mathbb{D}$. 
The gauge invariance of \eqref{eq:ymh1}--\eqref{eq:ymh2} implies that we will only be able to reconstruct the Yang--Mills--Higgs fields up to gauge. Hence, we shall consider the moduli space of the Yang--Mills--Higgs fields invariant under the action of the pointed gauge group \[G^0(\mathbb{D}, p) = \{\mathbf{U} \in C^\infty (\mathbb{D}; G) : \mathbf{U}(p) = \id\}, \quad \mbox{for $p = (-1, 0) \in \mathbb{D}$}.\] 
(The reason for requiring that $\mathbf{U}(p)=\id$ is technical in nature.)
For 
	\begin{align*}
(A, \Phi), (B, \Xi) \in \Omega^{1}(\mathbb{D},\g) \times C^\infty(\mathbb{D}, \mathcal{W}),
	\end{align*}
we say $(A, \Phi) \sim (B, \Xi)$ if there exists $\mathbf{U} \in G^0(\mathbb{D}, p)$ such that $(B, \Xi) = (A, \Phi) \cdot \mathbf{U}$ in $\mathbb{D}$.

The active local measurements that we perform to obtain the data for the problem have to account somehow for this gauge invariance. One possibility would be to attempt some gauge fixing in order to define a suitable source-to-solution map; in fact we shall do that at a later point, but in order to formulate the result it seems more appropriate to find a presentation of the data that is amenable to the gauge invariance. Having this in mind, we follow the approach we took in \cite{CLOP2} and define:
\begin{align}\label{def_mho}
	\mho := \{(t,x) : \text{$(t,x)$ is in the interior of $\mathbb D$ and $|x| < \varepsilon_0$} \},
\end{align}for a fixed $0 < \epsilon_0 < 1$, together with the following data set of $(A, \Phi)$, $$\mathcal{D}_{(A, \Phi)}: = \left\{\begin{array}{l|l}
	(V, \Psi)|_\mho & 
	\left.\begin{array}{l}
		\mbox{$(V, \Psi) \in C^3(\mathbb{D})$ solves \eqref{eq:ymh1}-\eqref{eq:ymh2} in $\mathbb{D} \setminus \mho$} \\ \mbox{and $(V, \Psi) \sim (A, \Phi)$ near $\partial^- \mathbb{D}$} \end{array} \right.  \end{array} \right\}.$$
Here
\[	\partial^- \mathbb{D}=\{(t,x)\in\mathbb{D}:\,|x|=t+1\}\]
and 	
$(V, \Psi) \sim (A, \Phi)$ near $\partial^- \mathbb{D}$ if there are $\mathbf{U} \in G^0(\mathbb{D}, p)$ and a neighbourhood $\mathcal  U$ of $\partial^- \mathbb{D}$ such that $(V,\Psi) = (A,\Phi)\cdot \mathbf{U}$ on $ \mathcal {U}\cap \mathbb{D}$.
Note that $p = (-1, 0) \in \bar{\mho}$.

Under suitable assumptions on $\g$ and $\rho_{*}$, we shall prove that one can uniquely determine the Yang--Mills--Higgs field $(A,\Phi)$ in $\mathbb{D}$ from the local information $\mathcal{D}_{(A,\Phi)}$ in $\mho$. Let $Z(\g)$ denote the centre of the Lie algebra $\g$.

\begin{theorem}\label{thm:main thm} Assume that the centre $Z(\g)$ is non-trivial,  $\rho$ is fully charged and 
	\begin{align*}
Z(\g)\cap \text{\rm Ker}\,\rho_{*}=\{0\}.
	\end{align*}
	Suppose $(A, \Phi)$ and $(B, \Xi)$ are two Yang--Mills--Higgs fields in $\mathbb{D}$.
	Then $$\mathcal{D}_{(A, \Phi)} = \mathcal{D}_{(B, \Xi)} \Longleftrightarrow (A, \Phi) \sim (B, \Xi) \, \mbox{in $\mathbb{D}$}.$$
	In other words,	 the data sets agree $\mathcal{D}_{(A, \Phi)} = \mathcal{D}_{(B, \Xi)}$ if and only if there exists a gauge transformation $\mathbf{U}\in G^0(\mathbb{D}, p)$ such that the following holds  in $\mathbb{D}$:
\[(B, \Xi) = (\mathbf{U}^{-1}d\mathbf{U}+\mathbf{U}^{-1}A\mathbf{U},\rho(\mathbf{U}^{-1})\Phi).\]
	 
\end{theorem}

\begin{remark}\label{rem_ss_map} Intuitively, it is helpful to think of the data set as produced by an observer creating sources $(J,\mathcal F)$ supported in $\mho$ and observing solutions $(V,\Psi)$ in $\mho$ to the coupled system
\begin{align}
  	&D_{V}^*F_{V}+\J_{\rho}(d_{V}\Psi,\Psi)=J,\label{eq:VJ}\\
  	&d_{V}^*d_{V}\Psi +\mathcal{V}'(|\Psi|^{2})\Psi=\label{eq:PsiF}\mathcal F.
  \end{align} 
  The first step in the proof of Theorem \ref{thm:main thm} consists in converting the data set into a source-to-solution map where the intuition above is rigorously realised. For this, we will follow \cite{CLOP2} and employ two different gauges: the temporal gauge and the relative Lorenz gauge.
  No new difficulty arises here in relation to \cite{CLOP2}.

 \end{remark}

\begin{remark} The hypotheses on $\g$ and $\rho_{*}$ are general enough to cover the cases that are of most interest to us. Let us start with $G=SU(2)\times S^1$ and $\rho$ as in the electroweak model.
In this instance, $\mathcal W=\mathbb{C}^2$ with the standard Hermitian
inner product and $\rho$ is given by
\[\rho(g,e^{\imath\theta})w=e^{\imath n_{Y}\theta}gw,\]
where $n_{Y}$ is an integer usually taken as $3$ and $(g,e^{\imath\theta})\in SU(2)\times S^{1}$. (Here $\imath = \sqrt{-1}$.) The derivative $\rho_{*}$ is easily computed as
\[\rho_{*}(X,\imath x)w=(X+\imath n_{Y}x\id)w\]
where $(X,\imath x)\in \g=\mathfrak{su}(2)\times \imath\mathbb{R}$. For $n_{Y}\neq 0$ we see that $\rho_{*}$ is fully charged and  $\text{\rm Ker}\,\rho_{*}=\{0\}$. The centre of the Lie algebra is non-trivial and equal to $\mathfrak{u}(1)=\imath\mathbb R$, the Lie algebra of the $S^1$-factor.

For the Standard Model we need to enlarge the group by adding an $SU(3)$-factor, that is\footnote{Strictly speaking the structure group of the Standard Model is $G/\Gamma$, where $\Gamma\subset \mathbb{Z}_{6}$ is some unknown subgroup \cite{Tong2017}. This is inconsequential for us as all these groups share the same Lie algebra.}, $G=SU(3)\times SU(2)\times S^{1}$. 
In this case the $SU(3)$-component acts trivially on the Higgs vector space $\mathcal W$ and this has the effect
of making $\text{Ker}\,\rho_{*}=\mathfrak{su}(3)$ but it remains fully charged.
Moreover, the hypotheses in the theorem are still satisfied since $\mathfrak{su}(3)\cap \mathfrak u(1)=\{0\}$.

In forthcoming work we shall complete the Standard Model (at the classical level) by including matter fields.

\end{remark}

 \subsection{Outline of the proof} \label{subsection:outline}We now provide an overview of the proof of Theorem~\ref{thm:main thm}.
 
\begin{itemize} 
	\item As in \cite{CLOP, CLOP2} we consider the non-linear interaction of three singular waves produced by sources which are conormal distributions following the approach pioneered in \cite{KLU}.

\item The sources $(J, \mathcal{F})$ in $\mho$ have the general form 
$$J = \sum_{k=1}^3\epsilon_{(k)} J_{(k)}$$
and 
$$\mathcal{F} = \sum_{k=1}^3\epsilon_{(k)} \mathcal{F}_{(k)}$$
for some small coefficients $\epsilon_{(k)}$. The source $J$ acts on the YM-channel (i.e. equation \eqref{eq:VJ}) while the source $\mathcal F$ activates the Higgs channel (i.e. equation \eqref{eq:PsiF}). When $J$ takes values in $Z(\g)$ we shall say that it is an EM-source (for electromagnetic).

\item The first step consists in choosing sources that do not activate the Higgs channel, i.e. $\mathcal F=0$, but
fully activate the YM-channel by letting $J = \sum_{k=1}^3\epsilon_{(k)} J_{(k)}$.
Arguing very closely to our previous work in \cite{CLOP2} we shall
be able to recover the broken non-abelian light ray transform of $A$ associated with the adjoint representation, $\mathbf{S}_{z\gets y\gets x}^{A, \text{Ad}}$.
There is not much new here in the sense that all the hard work was done in \cite{CLOP2}. Still, some care is needed as we are now dealing with coupled equations for $(A,\Phi)$.

\item The second step is the most interesting one and takes advantage that we have a Lie algebra with non-trivial centre and a representation that is fully charged. We choose the sources that produce the minimal disturbance in the system but that activate both channels. This means two sources in the EM-channel but just one in the Higgs channel. Specifically we take
$$J = \epsilon_{(2)} J_{(2)}+\epsilon_{(3)} J_{(3)},\;J_{(2)},J_{(3)}\in Z(\g)$$
and 
$$\mathcal{F} = \epsilon_{(1)} \mathcal{F}_{(1)}.$$
We will prove that with this choice, the triple interaction of the three singular waves allows the recovery
of the broken non-abelian light ray transform of $A$ in the representation $\rho$, $\mathbf{S}_{z\gets y\gets x}^{A, \rho}$, together with a weighted integral transform of $\Phi$. The weight is determined by $A$.

\item To obtain $A$ we just note that the direct sum representation $\text{Ad}\oplus\rho$ has a derivative that is faithful (i.e. trivial kernel) due to the hypothesis $Z(\g)\cap \text{\rm Ker}\,\rho_{*}=\{0\}$. Hence $A$ may also be recovered through the inversion of the broken non-abelian light ray in the fundamental representation, performed in \cite[Theorem 5]{CLOP}. Moreover, if $A$ is known, $\Phi$ is easily recovered from the weighted integral transform of $\Phi$ obtained in the second step.

\end{itemize}

\subsection{Comparison with existing literature} 

\begin{figure}
\centering
\includegraphics[width=0.5\textwidth]{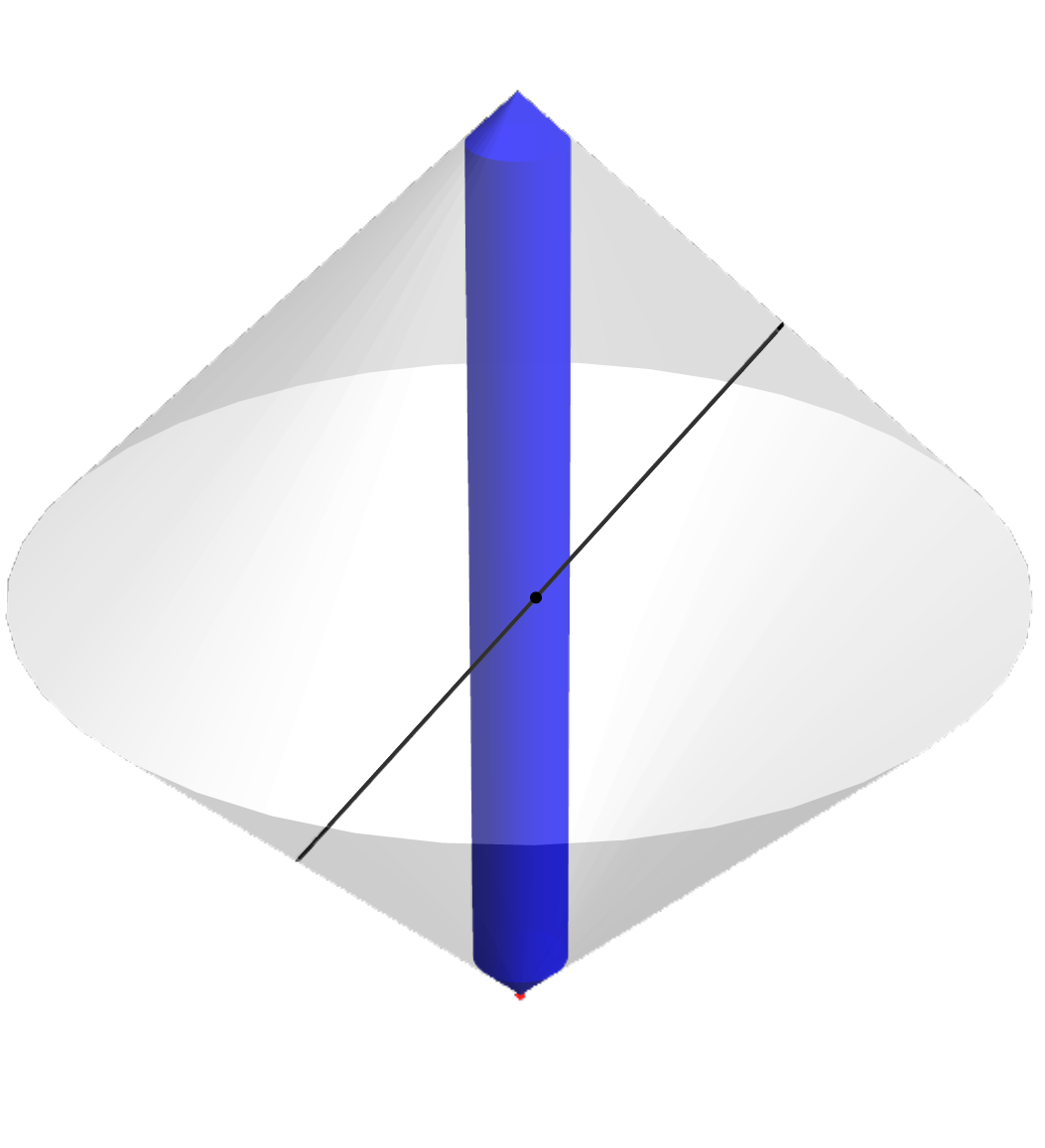}
\caption{The set $\mho$ (in blue) inside the diamond $\mathbb D$ in the $1+2$ dimensional case. The black line is a lightlike geodesic segment that intersects $\p \mho$ tangentially at the black point. The geodesic stays outside $\mho$ and this shows that the surface $\p \mho$ is not pseudoconvex. The point $p$ is drawn in red.}
\label{fig_D}
\end{figure}

The present work is a continuation of \cite{CLOP2} where we considered the pure Yang-Mills system. The direct theory in Section~\ref{section:perturbedYMH} is parallel to that in \cite{CLOP2}, and on the very basic level, so is the approach to solve the inverse problem via a threefold linearization. 

The coupled system (\ref{eq:ymh1})-(\ref{eq:ymh2}) is more intricate than the pure Yang-Mills system. From the point of view of the approach in \cite{CLOP2} this leads to two essential complications. First, the non-abelian ray transform associated to the latter system is simply the parallel transport given by the Yang-Mills connection, whereas for the former it is given by a coupled system of transport equations, see (\ref{eq:imp})-(\ref{eq:imp2}) below. Second, the algebraic computations related to principal symbols of the terms generated by the threefold linearization become more challenging than in \cite{CLOP2}. 

Nonetheless, we show that the Yang-Mills field in the coupled system can be recovered via a reduction to the pure Yang-Mills case, on the level of principal symbols. 
The Higgs field is not present in \cite{CLOP2}. Here we recover it using a new argument based on choosing sources for the Yang-Mills field that take values in the center $Z(\mathfrak g)$. This can be viewed as a key novelty in the paper as it makes the algebraic computations related to the the principal symbols manageable. 

The technique to utilize multiple linearizations originates from \cite{KLU}, where it was applied to recover the leading order terms in a semilinear wave equation with a simple power type nonlinearity. 
Subsequently, lower order terms and coefficients in nonlinear terms were recovered as well, see \cite{Feizmohammadi2022} and \cite{LUW}. Multiple linearizations have been used also to solve inverse problems for elliptic \cite{Feizmohammadi2020, Lassas2021}
and real principal type operators (containing both wave and elliptic operators as special cases) \cite{Oksanen2020}. 
We will not attempt to give a detailed review of the ever-growing body of literature using the technique. An early survey of the theory is given in \cite{Matti}.

The above mentioned inverse problems for nonlinear real principal type operators, and their special cases, differ from the problem in the present paper in two important ways. In the real principal type case, the direct problem is not gauge invariant and the unknowns to be recovered are coefficients in the equation. 
In the case of the Yang-Mills-Higgs equations, on the other hand, there is a gauge invariance and the unknowns are solutions to the equations, not independent coefficients. Inverse problems for the Einstein equations \cite{KLOU, Lassas2017a, Uhlmann2018} share these two features. 

In fact, the inverse problems for the Yang-Mills-Higgs and Einstein equations can be seen to be closer to unique continuation than coefficient determination problems. However, they cannot be solved using classical unique continuation results \cite{Hormander-Vol4}, since these are known to fail in geometric settings lacking pseudoconvexity \cite{Alinhac1983}. Unique continuation from $\mho$ to $\mathbb D$ is an example of such a setting, see Figure~\ref{fig_D}. The unique continuation results not requiring pseudoconvexity, \cite{Tataru} being a prime example, impose analyticity conditions, and thus are also unsuitable for our purposes. 

Finally, let us mention that the inverse problem associated to the linearized version of the Yang-Mills-Higgs system is open. In fact, time-dependent, matrix-valued first and zeroth order coefficients in linear vector-valued wave equations have not been recovered even in the Minkowski geometry; only the time-independent, symmetric case has been solved \cite{KOP}. 
Finally, we mention that the analogous problem for linear elliptic equations is open (see \cite{Cekic2017a, Cekic2020} for the closest positive results), but for the linear dynamical Schr\"odinger equation a solution is given in \cite{Tetlow2022}.

\subsection{Structure of the paper}
\begin{itemize}
\item Section \ref{section:lightray} discusses connection wave operators and the broken non-abelian light ray transform of a connection for an arbitrary representation $\rho$. Corollary \ref{corollary:rhofaithful} shows that if $\rho_{*}$ is faithful, then one can recover the connection up to gauge from the broken non-abelian light ray transform. We finish the section with a discussion of the parallel transport that arises for pairs $(A,\Phi)$ when considering the subprincipal symbol of the linearized Yang--Mills--Higgs equations.
\item Section \ref{section:perturbedYMH} considers the Yang--Mills--Higgs equations with sources. We use the temporal gauge and the relative Lorenz gauge to formulate a suitable source-to-solution map.
\item Section \ref{sec:linearization} computes the equations for the triple cross-derivative when three sources are introduced. The focus here is on the leading interaction terms. In this section we also set up the technical aspects of the sources being used.
\item Section \ref{section:recoveryad} explains how to recover the broken non-abelian light transform of the connection in the adjoint representation.
\item Section \ref{section:abelian} contains the key microlocal calculation that allows for the recovery of the broken light ray transform of $A$ in the representation $\rho$ and a weighted integral transform of $\Phi$ whose weight is determined by $A$. The proof of Theorem \ref{thm:main thm} is completed here.

\end{itemize}

  
 
\section{
Light ray transforms and wave operators}\label{section:lightray}
  
\subsection{The parallel transport}
 
Given a connection $A \in \Omega^1(M, \g)$, we can introduce the parallel transport $\mathbf{U}_\gamma^A$ on the principal bundle $M \times G$ along a curve $\gamma: [0, T] \rightarrow M$. Namely, $\mathbf{U}_\gamma^A = u(T)$ is given by the solution to the following ODE
 \begin{align*}
 	\dot{u} + A_{\gamma}(\dot{\gamma}) \, u &= 0\\
 	u(0) &= \id.
 \end{align*} 
 Here $A$ is viewed as a covector, $\dot{\gamma}(t)$ is deemed as a vector, and $A_{\gamma}(\dot{\gamma})$ denotes the pairing $\langle A, \dot{\gamma}(t) \rangle$.
Moreover, given a vector space $\mathbb V$ (real or complex) and a linear representation $\rho:G\to \GL(\mathbb{V})$, we obtain a parallel transport acting on $\mathbb{V}$ by setting $\mathbf{P}_\gamma^{A, \rho} = \rho(\mathbf{U}_\gamma^A)$.
Alternatively, one may define the parallel transport $\P_\gamma^{A, \rho}$ directly through the parallel transport equation \begin{equation} 	\label{eqn : parallel transport for rho}
  \begin{aligned}
 \dot{v} + \rho_\ast(A_{\gamma}(\dot{\gamma})) v &= 0\\
	v(0) &= w.  
\end{aligned} 
\end{equation}
Since $G$ is a matrix Lie group, we have $G\subset \GL(n,\C)$ for some $n$; in this case if $\rho=\text{inc}$ is just the inclusion, then $\mathbf{P}_\gamma^{A, \text{inc}}v = \mathbf{U}_\gamma^A v$ for all $v\in \C^n$.

Recall that the adjoint representation is defined by $$\Ad : G \rightarrow \text{Aut}(\g),\quad h \mapsto \Ad_h, \quad$$
with $\Ad_h w = h w h^{-1}$ for $w \in \g$. 
It follows in this case that $\mathbb{V}=\g$ and $$\mathbf{P}_\gamma^{A, \Ad} w = \Ad_{\mathbf{U}_\gamma^A} w =  \mathbf{U}_\gamma^A w (\mathbf{U}_\gamma^A)^{-1}\quad\mbox{for $w \in \g$}.$$ Equivalently, $v(t) = \mathbf{U}_\gamma^A w (\mathbf{U}_\gamma^A)^{-1}$ solves the following ODE, \begin{equation} 	\label{eqn : parallel transport for ad}\begin{aligned}
	\dot{v} + [ A_{\gamma}(\dot{\gamma}), v] &= 0\\
	v(0) &= w. 
\end{aligned} \end{equation}
 
 For any two points  $x, y \in \mathbb{D} \subset M$, there is a unique geodesic $\gamma$ from $x$ to $y$. Since the parallel transport $\mathbf{P}_\gamma^{A, \rho}$ is independent of the choice of parametrization of $\gamma$, we can define the parallel transport from $x$ to $y$ by $$\mathbf{P}_{y \gets x}^{A, \rho} = \mathbf{P}_\gamma^{A, \rho}.$$

\subsection{The connection wave operator}

For a fixed connection $A$, the second equation \eqref{eq:ymh2} in the Yang--Mills--Higgs system is a nonlinear hyperbolic equation for $\Phi$.
Its linear part is given by the connection wave operator
$\Box_{A,\rho} = d_A^* d_A$.
It turns out that the first equation \eqref{eq:ymh1} has a very similar structure, in a suitable gauge, but to keep the notation simple, we will focus on equations of the form
	\begin{align}\label{connection_wave}
\Box_{A,\rho} \Phi = 0
	\end{align}
for a moment. 
 
Connection wave operators fit into the framework of differential operators of real principal type, and the corresponding fundamental solutions are given by Fourier integral operators. 
There holds, modulo zeroth order terms,
    \begin{align*}
\Box_{A,\rho} \Phi 
=
-\p^\alpha \p_\alpha \Phi
-2 \rho_*(A_\alpha) \p^\alpha \Phi,
    \end{align*}
see e.g.   \cite[Section 2.1]{CLOP}.
Here the indices are raised and lowered with respect to the Minkowski metric.
Thus the principal and subprincipal symbols of $\Box_{A,\rho}$ are
    \begin{align}\label{sigma_Box}
   	\sigma[\Box_{A, \rho}](x, \xi) = \xi^\alpha \xi_\alpha, \quad \sigma_{\mathrm{sub}}[\Box_{A, \rho}](x,\xi) = 2 \imath^{-1}  \xi^\alpha \rho_\ast (A_\alpha).
   \end{align}
As the principal symbol does not depend on $A$ or $\rho$
we write simply $\sigma[\Box_{A, \rho}] = \sigma[\Box]$.

It is classical, see \cite{Duistermaat-Hormander-FIO2}, that 
if a conormal distribution $\Upsilon$ solves \eqref{connection_wave} then its principal symbol $\sigma[\Upsilon]$
satisfies a transport equation associated to $\sigma[\Box]$ and $\sigma_{\mathrm{sub}}[\Box_{A, \rho}]$. 
This transport equation can be reduced to the parallel transport equation \eqref{eqn : parallel transport for rho}, see e.g.   \cite[Section 2.6]{CLOP}. If $K \subset M$ is the submanifold on which 
$\Upsilon$ is singular, then $K$ is of codimension one and the conormal bundle $N^* K$ is contained in the characteristic variety 
	\begin{align*}
\{(x,\xi) \in T^* M : \sigma[\Box](x,\xi) = 0\}.
	\end{align*}
We may trivialize the half density bundle over $N^* K$
by choosing a strictly positive half density $\omega$, and it is convenient to choose $\omega$ so that it is positively homogeneous of degree 1/2. 
Let $\gamma(t) = t \xi + x$ be a lightlike line such that the segment from $\gamma(0) = x$ to $\gamma(s) = y$ is contained in $K$
for some $s \in \R$. 
We
recall that $\gamma$ being lightlike means that $\xi^\alpha \xi_\alpha = 0$.
Let us consider the rescaled principal symbol of $\Upsilon$ along $\gamma$,
\begin{align}\label{prin_symb_rescaled}
	e^{\varrho(t)}\left(\omega^{-1}\sigma[\Upsilon]\right)   (\gamma(t), \dot \gamma_\alpha(t) dx^\alpha),
\end{align} 
where 
\begin{align}\label{def_varrho}
\varrho(t) = \int_{0}^t \left(\omega^{-1}\mathscr{L}_{H_{\sigma[\Box]}} \omega\right)\left(\gamma(r), \dot \gamma_\alpha(r) dx^\alpha\right) dr.
\end{align} 
Here $\mathscr{L}_{H_{\sigma[\Box]}}$ is the Lie derivative with respect to the Hamilton vector field $H_{\sigma[\Box]}  = 2\xi^\alpha \p_\alpha$ associated to the principal symbol $\sigma[\Box]$, acting on half densities.
Now we can express the evolution of \eqref{prin_symb_rescaled}
as follows
\begin{align}
\label{pt_symbol_rho}
 		e^{\varrho(s)}(\omega^{-1}\sigma[\Upsilon])(y, \xi) &=  \P_{y \gets x}^{A,\rho} \left((\omega^{-1}\sigma[\Upsilon])(x, \xi)\right),
 \end{align} 
where $y = \gamma(s)$.

 \subsection{The broken light ray transform}
\label{sec_S}
 As was seen in \cite{CLOP, CLOP2}, the multiple linearization scheme leads to parallel transport along two lightlike geodesic segments. This type of parallel transport over broken geodesics is termed broken light ray transform. To formalize this notion, we introduce
 \begin{align}\label{def_diamonds_etc}
 	\mathbb L &= \{(x,y) \in {\mathbb D}^2 : \text{there is a lightlike line joining $x$ and $y$}\},
 	\\ \label{def_broken_triplet}
 	\mathbb S^+(\mho) &= \{(x,y,z) \in {\mathbb D}^3 : (x,y), (y,z) \in \mathbb L,\ x < y < z,\ x,z \in \mho,\ y \notin \mho \},
 \end{align}
 where $x<y$ means that there is a future pointing causal curve from $x$ to $y$.
 (For $(x,y) \in \mathbb L$, we have $x<y$ if and only if the time coordinate of $y-x$ is strictly positive.) 
Let $A\in \Omega^{1}(\mathbb{D},\g)$ and $\rho:G\to\GL(\mathbb{V})$ a linear representation.
The broken light ray transform $\mathbf{S}_{z \gets y \gets x}^{A, \rho}$ with respect to $(A, \rho)$ over the broken future pointing geodesic via $(x, y, z) \in \mathbb{S}^+(\mho)$ is defined as the parallel transport travelling from $x$ through $y$ to $z$, that is $$\mathbf{S}_{z \gets y \gets x}^{A, \rho}:= \mathbf{P}_{z \gets y}^{A, \rho} \mathbf{P}_{y \gets x}^{A, \rho}.$$

\begin{theorem} If 
\[\mathbf{S}^{A,\rho}_{z \gets y \gets x}=\mathbf{S}^{B,\rho}_{z \gets y \gets x}\]
for all $(x,y,z)\in  \mathbb S^+(\mho)$, then 
there exists a smooth $\mathbf{U}:\mathbb{D}\to \rho(G)$ such that $\mathbf{U}|_{\mho}=\id$ and
$\rho_{*}(B)=\mathbf{U}^{-1}d\mathbf{U}+\mathbf{U}^{-1}\rho_{*}(A)\mathbf{U}$.
\label{thm:main2}
\end{theorem}

\begin{proof} We have previously inverted the transform $\mathbf{S}_{z \gets y \gets x}^{A, \text{inc}}$ in the case of $G = U(n)$, see \cite[Theorem 5]{CLOP}, where a slightly different choice of $\mho$ and $\mathbb{D}$ is used.
However, the proof works for any matrix Lie group or any $\rho(G)\subset \GL(\mathbb{V})$, and also for the present choice of $\mho$ and $\mathbb{D}$. Moreover, the gauge $u$ defined in \cite[Lemma 3]{CLOP} is smooth up to $\partial\mathbb D$ whenever the two connections $A$ and $B$ are smooth up to $\partial \mathbb{D}$.
\end{proof}

Since we are interested in the full recovery of the connection, we now give a simple condition on $\rho_{*}$ that allows for this recovery:

\begin{corollary} Suppose $\rho_{*}$ is faithful, i.e. $\text{\rm Ker}\,\rho_{*}=\{0\}$. If
\[\mathbf{S}^{A,\rho}_{z \gets y \gets x}=\mathbf{S}^{B,\rho}_{z \gets y \gets x}\]
for all $(x,y,z)\in  \mathbb S^+(\mho)$, then 
there exists a smooth $\mathbf{U}:\mathbb{D}\to G$ such that $\mathbf{U}|_{\mho}=\id$ and
$B=\mathbf{U}^{-1}d\mathbf{U}+\mathbf{U}^{-1}A\mathbf{U}$.
\label{corollary:rhofaithful}
\end{corollary}

\begin{proof} If $\rho_{*}$ is faithful, the map $\rho:G\to\rho(G)$ is a covering of Lie groups ($\rho_{*}$ becomes an isomorphism from $\g$ to the Lie algebra of $\rho(G)$). Since $\mathbb{D}$ is simply connected
there is a unique smooth lift of $\mathbf{U}:\mathbb{D}\to\rho(G)$ to a map
$\tilde{\mathbf{U}}:\mathbb{D}\to G$ such that $\tilde{\mathbf{U}}|_{\mho}=\id$. Now it is easy to check that
$\mathbf{U}$ satisfies $\rho_{*}(B)=\mathbf{U}^{-1}d\mathbf{U}+\mathbf{U}^{-1}\rho_{*}(A)\mathbf{U}$ iff $\tilde{\mathbf{U}}$ satisfies
$B=\tilde{\mathbf{U}}^{-1}d\tilde{\mathbf{U}}+\tilde{\mathbf{U}}^{-1}A\tilde{\mathbf{U}}$.
Indeed, if we let $\omega$ and $\tilde{\omega}$ be the (left) Maurer--Cartan 1-form of $\rho(G)$ and $G$ respectively, then the gauge equivalences may be expressed as 
\[\rho_{*}(B)=\mathbf{U}^*\omega+\text{Ad}_{\mathbf{U}^{-1}}(\rho_{*}(A))\]
and 
\[B=\tilde{\mathbf{U}}^*\tilde{\omega}+\text{Ad}_{\tilde{\mathbf{U}}^{-1}}(A).\]
Using that $\rho^*\omega=\rho_{*}(\tilde{\omega})$, $\rho\circ \tilde{\mathbf{U}}=\mathbf{U}$, 
$\rho_{*}\circ \text{Ad}_{\tilde{\mathbf{U}}^{-1}}=\text{Ad}_{\mathbf{U}^{-1}}\circ\rho_{*}$ (cf. \cite[Proposition 2.1.46]{Ha}) and that $\rho_{*}$ is an isomorphism between the Lie algebra of $G$ and that of $\rho(G)$, the equivalence between the two gauge equivalences follows.
\end{proof}

\begin{remark} One consequence of the argument presented above is that the data set $\mathcal D_{(A,\Phi)}$ does not really depend on the group $G$ as long as it has Lie algebra $\g$.
\end{remark}

\subsection{The parallel transport coupled with Higgs fields}
 
Let us now consider the system \eqref{eq:VJ}--\eqref{eq:PsiF}
for $(V,\Psi)$
and write 
	\begin{align*}
W = V - A, \quad \Upsilon = \Psi - \Phi,
	\end{align*}
where $(A,\Phi)$ is assumed to solve the corresponding homogeneous system \eqref{eq:ymh1}--\eqref{eq:ymh2}.
It turns out that $(W, \Upsilon)$ solves then a system 
with the linear part given by 
    \begin{align*}
P \begin{pmatrix}
W
\\
\Upsilon
\end{pmatrix}
=
\begin{pmatrix}
\Box_{A,\Ad} W + \mathbb J_\rho(d \Upsilon, \Phi)
\\
\Box_{A,\rho} \Upsilon 
\end{pmatrix},
    \end{align*}
whenever $W$ satisfies the relative gauge condition $D_A^* W = 0$. This is discussed in detail in Section \ref{sec_rel_lor} below.
Here we will study the linear operator $P$. 

The operator $P$ has the principal and subprincipal symbols given by (cf. \eqref{sigma_Box}), 
    \begin{align*}
\sigma[P](x,\xi) = \xi^\alpha \xi_\alpha,
\quad
\sigma_{sub}[P](x,\xi)
\begin{pmatrix}
w \\ \upsilon
\end{pmatrix}
= 
\imath^{-1} \begin{pmatrix}
2[\xi^{\alpha}A_\alpha, w_{\beta}]dx^{\beta} -\xi_{\beta} \mathbb J_\rho(\upsilon, \Phi)dx^{\beta} 
\\
2\rho_*(A_\alpha) \xi^\alpha \upsilon
\end{pmatrix},
    \end{align*}
and the associated parallel transport equation along a lightlike line $\gamma$ reads
 \begin{align} 	\dot{w}_{\beta}+\left[A_{\gamma}(\dot{\gamma}),w_{\beta}\right] - \frac12\dot{\gamma}_{\beta}(t)\,\J_{\rho}\left(\upsilon,\Phi(\gamma(t))\right)&=0,\label{eq:imp}\\
 	\dot{\upsilon}+\rho_{*}(A_{\gamma}(\dot{\gamma}))\upsilon&=0, \label{eq:imp2}
 \end{align} for $\beta = 0, 1, 2, 3$. 
  If $\gamma:[0,T]\to M$, then $w_{\beta}:[0,T]\to \g$ for all $\beta$ and $\upsilon:[0,T]\to \mathcal{W}$. The solutions $(w_{\beta}(t), \upsilon(t))$ to \eqref{eq:imp}--\eqref{eq:imp2} are uniquely determined by fixing initial conditions $(w_{\beta}^0, \upsilon^0)$.
 The second equation is decoupled and its fundamental solution is the parallel transport
 $\P_\gamma^{A,\rho} = \rho(\U_\gamma^A)$. Indeed we have
 $\upsilon(t)=\P_\gamma^{A,\rho}\upsilon^0$. By Duhamel's principle, the solution $w_{\beta}(t)$ to \eqref{eq:imp} with initial conditions $(w_{\beta}^0, \upsilon^0)$ is given by
 	\begin{equation}\label{eqn : solution to the transport (YM)}w_{\beta}(t)=\P_\gamma^{A,\text{\rm Ad}}(t)w_{\beta}^{0}   +   \frac12  \P_\gamma^{A,\text{\rm Ad}}(t)\int_{0}^{t}(\P_\gamma^{A,\text{\rm Ad}}(s))^{-1} \dot{\gamma}_{\beta}(s)\J_{\rho}\left(\upsilon(s),\Phi(\gamma(s))\right)\,ds.\end{equation}

Since $\P_\gamma^{A,\text{\rm Ad}}=\text{\rm Ad}_{\U^{A}_{\gamma}}$, we employ the following equivariance property of the coupling product $\J_\rho$ to simplify this formula.
 \begin{lemma} 	For all $g\in G$ and $v,w\in \mathcal{W}$
 	\[\text{\rm Ad}_{g}(\J_{\rho}(v,w))=\J_{\rho}(\rho(g)v,\rho(g)w).\]
 
 	\label{lemma:equivariance}
 \end{lemma}
 \begin{proof} For any $X\in\g$ we have, using invariance of the inner products,
 	\begin{align*}
 		&(\text{\rm Ad}_{g}(\J_{\rho}(v,w)), X)_{\text{\rm Ad}}=(\J_{\rho}(v,w), \text{\rm Ad}_{g^{-1}}X)_{ \text{\rm Ad}}
 		=\Re (v,\rho_{*}(\text{\rm Ad}_{g^{-1}}(X))w)_{ \mathcal{W}}\\&\qquad
 		=\Re (v,\text{\rm Ad}_{\rho(g^{-1})}(\rho_{*}(X))w)_{ \mathcal{W}}
		=\Re (v,\rho(g^{-1})\rho_{*}(X)\rho(g)w)_{ \mathcal{W}}
		\\&\qquad
 		=\Re (\rho(g)v,\rho_{*}(X)\rho(g)w)_{ \mathcal{W}}
 		=(\J_{\rho}(\rho(g)v,\rho(g)w),X)_{ \text{\rm Ad}},
 	\end{align*}
 	as desired (in the third equality, we used \cite[Proposition 2.1.46]{Ha}).
 	\end{proof}
It follows that \eqref{eqn : solution to the transport (YM)} is reduced to
 \begin{equation}
 	w_{\beta}(t)=\P_\gamma^{A,\text{\rm Ad}}(t)w_{\beta}^{0}+ \frac12 \P_\gamma^{A,\text{\rm Ad}}(t)\int_{0}^{t}\dot{\gamma}_{\beta}(s)\J_{\rho}(\upsilon^0,\rho(\U^{A}_{\gamma}(s))^{-1}\Phi(\gamma(s)))\,ds.
 	\label{eq:wbeta}
 \end{equation}
 If $x=\gamma(0)$ and $y=\gamma(T)$ we define the following linear transformation
 \[\mathbf{P}_{y\gets x}^{A,\Phi,\rho} (w_{\beta}^{0},\upsilon^0):= (w_{\beta}(T),\upsilon(T)).\]
One can check that this map is independent of the parametrization chosen for $\gamma$ thanks to the presence of the term $\dot{\gamma}_{\beta}$ in \eqref{eq:wbeta}. We note that $\mathbf{P}_{y\gets x}^{A,\Phi,\rho} $ depends on $\beta$ even though we do not indicate this explicitly in order to simplify the notation.

Using \eqref{eq:wbeta}, we may represent $\P_\gamma^{A, \Phi, \rho}$ from $x = \gamma(t_1)$ to $y = \gamma(t_2)$ in terms of an upper triangular block matrix, 
\begin{equation*}
	\mathbf{P}_{y\gets x}^{A, \Phi, \rho} = \left( 
	\begin{array}{cc}
		\P_{y\gets x}^{A,\text{\rm Ad}}   & \left(\mathbf{P}_{y\gets x}^{A, \Phi, \rho}\right)_{12}
		\\ & \\
		0 & \P_{y\gets x}^{A, \rho}   
	\end{array}\right), 
\end{equation*} where the $(1,2)$-entry of $\mathbf{P}_{y\gets x}^{A, \Phi, \rho}$ operates as
\[ \begin{aligned} & \left( 	\P_{y\gets x}^{A,\Phi,\rho}   \right)_{12} :   \mathcal W \to\g  \\  & \left( 	\P_{y\gets x}^{A,\Phi,\rho}   \right)_{12}(\upsilon) =	\frac12\P_{y\gets x}^{A,\text{\rm Ad}} \int_{t_1}^{t_2}\dot{\gamma}_{\beta}(s) \J_{\rho}(\upsilon ,\rho(\U^{A}_{\gamma}(s))^{-1}\Phi(\gamma(s)))\,ds. \end{aligned} \]

\begin{remark} For each $\beta$ fixed, it is possible to interpret the linear transformation $\mathbf{P}_{y\gets x}^{A, \Phi, \rho}$ as the parallel transport of a suitable connection constructed from $A$ and $\Phi$. Indeed consider the $\text{End}(\g\oplus \mathcal W)$-valued connection $\mathbb{A}$ given by
\[\mathbb{A}_{x}(u)(X,\upsilon):=\left ([A_{x}(u),X]-\frac{1}{2}u_{\beta}\mathbb{J}_{\rho}(\upsilon, \Phi(x)), \rho_{*}(A_{x}(u))\upsilon\right )\in \g\oplus\mathcal W,\]
where $u\in T_{x}M$ and $(X,\upsilon)\in \g\oplus \mathcal W$. It is a simple exercise to check that $\mathbf{P}_{y\gets x}^{A, \Phi, \rho}=\mathbf{P}_{y\gets x}^{\mathbb{A}}$. This implies in particular that $\mathbf{P}_{y\gets x}^{A, \Phi, \rho}$ is independent of the parametrization of $\gamma$ as mentioned above.
\end{remark}
 
Suppose now that $P(W,\Upsilon) = 0$ and that $(W,\Upsilon)$ is a conormal distribution associated to a submanifold $K$. 
Let $\gamma(t) = t \xi + x$ be a lightlike line such that the segment from $\gamma(0) = x$ to $\gamma(s) = y$ is contained in $K$
for some $s \in \R$.
Then the principal symbol of $(W,\Upsilon)$ evolves along $\gamma$ analogously to (\ref{pt_symbol_rho}), as recorded in the following proposition. 
The proposition tells also that the evolution preserves homogeneity.

\begin{proposition}\label{prop_homogeneity}
If the principal symbol $\sigma[(W, \Upsilon)]$ is positively homogeneous of degree $q + 1/2$ at $(x,\xi)$ as a half density, that is,
$$\sigma[(W, \Upsilon)] (x, \lambda \xi) = \lambda^q \sigma[(W, \Upsilon)] (x, \xi) \quad \mbox{for any $\lambda > 0$},$$
then we have 
\HOX{Here $\mathbf{P}_{y\gets x}^{A, \Phi, \rho}$ acts on all components simultaneously. We need to fix the notations better.}
\begin{align}
	e^{\varrho(s)}(\omega^{-1}\sigma[(W,\Upsilon)])(y, \lambda \xi) = \lambda^q  \mathbf{P}_{y\gets x}^{A, \Phi, \rho} ((\omega^{-1}\sigma[(W,\Upsilon)])(x, \xi) ), \quad &\mbox{for any $\lambda > 0$}.
\end{align}
\end{proposition}
\begin{proof}
It follows from \cite{Duistermaat-Hormander-FIO2}
that the principal symbol $a = \sigma[(W, \Upsilon)]$ 
satisfies the transport equation
	\begin{align}\label{Lie_transport}
\mathscr L_{H_{\sigma[\Box]}} a + \imath \sigma_{\mathrm{sub}}[P] a = 0,
	\end{align}
see also \cite[Proposition 5.4]{Melrose-Uhlmann-CPAM1979}
and \cite[Theorem 3]{CLOP} for formulations that are close to the present case. 
We write 
	\begin{align*}
\Gamma : [0, s] \to N^* K, \quad
\Gamma(t) = (\gamma(t), \dot \gamma_\alpha(t) dx^\alpha),
	\end{align*}
and $(w, \upsilon) = e^\varrho \omega^{-1} a \circ \Gamma$,
where $\varrho$ is as in \eqref{def_varrho}.
Then the argument in \cite[Section 2.6]{CLOP} gives 
	\begin{align*}
 e^{\varrho} \omega^{-1} \mathscr L_{H_{\sigma[\Box]}} a \circ \Gamma
= 2\p_t 
\begin{pmatrix}
w \\ \upsilon
\end{pmatrix}.
	\end{align*}
Moreover,
	\begin{align*}
e^{\varrho} \omega^{-1} \imath \sigma_{\mathrm{sub}}[P] a \circ \Gamma
= 
\begin{pmatrix}
2[\dot \gamma^{\alpha}A_\alpha, w_{\beta}]dx^{\beta} -\dot \gamma_{\beta} \mathbb J_\rho(\upsilon, \Phi)dx^{\beta} 
\\
2\rho_*(A_\alpha) \dot \gamma^\alpha \upsilon
\end{pmatrix},
	\end{align*}
and we see that $(w, \upsilon)$ satisfies the transport equation \eqref{eq:imp}--\eqref{eq:imp2}. 
This shows the claim in the case that $\lambda = 1$.

Let us consider arbitrary $\lambda > 0$, and set $\gamma_\lambda(t) = t \lambda \xi + x$.
Then $\gamma_\lambda(0) = x$ and $\gamma_\lambda(s/\lambda) = y$.
We showed as a part of \cite[Proposition 1]{CLOP}
that $\varrho$ is independent of the parametrization of $\gamma$
in the sense that if $\varrho_\lambda$ is defined by \eqref{def_varrho}, with $\gamma$ replaced by $\gamma_\lambda$, then 
$\varrho_\lambda(s / \lambda) = \varrho(s)$.
Applying the above argument to $\gamma_\lambda$ gives
	\begin{align*}
e^{\varrho(s)}(\omega^{-1}\sigma[(W,\Upsilon)])(y, \lambda \xi) = \mathbf{P}_{y\gets x}^{A, \Phi, \rho} ((\omega^{-1}\sigma[(W,\Upsilon)])(x, \lambda\xi) ).
	\end{align*}
Here we used the fact that $\mathbf{P}_{y\gets x}^{A, \Phi, \rho}$ is independent of the parametrization of $\gamma$.
The claim follows from the assumed homogeneity of $\sigma[(W,\Upsilon)]$, together with the homogeneity of $\omega$ and linearity of $\mathbf{P}_{y\gets x}^{A, \Phi, \rho}$. 
\end{proof}

 \section{The perturbed Yang--Mills--Higgs system}\label{section:perturbedYMH} Suppose $(A,\Phi)$ is a background Yang--Mills--Higgs field on the fibre bundles $\Ad$ and $E$ over $M$ for a given $\rho$. We want to reconstruct $(A,\Phi)$ in the diamond $\mathbb{D}$ by doing active local measurements in $\mho$.
That is, we add sources $(J,\mathcal{F})$ with $J\in \Omega^{1}(M,\g)$ and $\mathcal{F}\in C^{\infty}(M,\mathcal{W})$, both supported in $\mho$, in the right-hand side of \eqref{eq:ymh1}--\eqref{eq:ymh2}. Then the Yang--Mills--Higgs field $(A, \Phi)$ will be perturbed by a quantity $(W, \Upsilon)$ and the perturbed field $(V, \Psi)$, with $V=A+W$ and $\Psi=\Phi+\Upsilon$, solves the system
 \begin{align}
 	&D_{V}^*F_{V}+\J_{\rho}(d_{V}\Psi,\Psi)=J;\label{eq:ymhs1}\\
 	&d_{V}^*d_{V}\Psi +\mathcal{V}'(|\Psi|^{2})\Psi=\mathcal{F}.\label{eq:ymhs2}
 \end{align} 
The active local measurements $\mathcal{D}_{(A, \Phi)}$ of the background Yang--Mills--Higgs field can be reduced to the source-to-solution map of the perturbed Yang--Mills--Higgs system \eqref{eq:ymhs1}--\eqref{eq:ymhs2}, which we shall define in this section.

 \subsection{Compatibility condition}
The sources $(J, \mathcal{F})$ in \eqref{eq:ymhs1}--\eqref{eq:ymhs2} cannot be arbitrarily chosen but must satisfy an additional compatibility condition. \begin{lemma}\label{lem_comp_cond}
 Suppose $(V, \Psi)$ and $(J, \mathcal{F})$ solve \eqref{eq:ymhs1}--\eqref{eq:ymhs2}. Then the sources $(J,\mathcal{F})$ satisfy the compatibility condition \begin{equation}D_{V}^*J=\J_{\rho}(\mathcal{F},\Psi).\label{eq:compatability for ymhs}\end{equation}
\end{lemma}

\begin{proof}

For a pair $(V, \Psi) \in \Omega^1(M, \g) \times C^\infty(M, \mathcal{W})$, the Yang--Mills--Higgs Lagrangian density $\mathcal{L}_{\text{YMH}}(V, \Psi)$, defined as in \eqref{eqn : YMH Lagrangian}, is gauge invariant, i.e. \[\mathcal{L}_{\text{YMH}}\left(V \cdot\mathbf{U}, \rho(\mathbf{U}^{-1})\Psi\right) = \mathcal{L}_{\text{YMH}}\left(V, \Psi\right).\]

We take a family of gauges    \[\left\{\mathbf{U}_s \in C^\infty(M, G) : \mathbf{U}_0 = \id, \, \partial_s \mathbf{U}_s |_{s=0} = \dot{\mathbf{U}} \in C_0^\infty(M, \g), \, s \in \mathbb{R}\right\}.\] 
It follows from the precise form of gauge transforms \eqref{eqn : gauge transform} that \begin{align*}
 	\partial_s (V \cdot \mathbf{U}_s) |_{s=0} &= d \dot{\mathbf{U}} + [V, \dot{\mathbf{U}}] = D_V \dot{\mathbf{U}}\\
 	\partial_s (\rho(\mathbf{U}_s^{-1}) \Psi) |_{s=0} &= - \rho_\ast(\dot{\mathbf{U}}) \Psi.
 \end{align*}
The gauge invariance of $\mathcal{L}_{\text{YMH}}(V, \Psi)$ implies that  \[\partial_s \mathcal{L}_{\text{YMH}}(V \cdot \mathbf{U}_s, \rho(\mathbf{U}_s^{-1})\Psi) \equiv 0,\] which, in terms of covariant derivatives, takes the form 
\[ ( D_V^\ast F_V + \mathbb{J}_\rho(d_V \Psi, \Psi), D_V \dot{\mathbf{U}} )_{L^2,\Ad} + \Re ( d_V^\ast d_V \Psi + \mathcal{V}'(|\Psi|^2) \Psi, - \rho_\ast(\dot{\mathbf{U}})\Psi )_{L^2, E} =0.\]
If $(V, \Psi)$ solve \eqref{eq:ymhs1}--\eqref{eq:ymhs2}, we have \[ ( D_V^\ast J, \dot{\mathbf{U}} )_{L^2,\Ad} -  ( \mathbb{J}_\rho(\mathcal{F}, \Psi), \dot{\mathbf{U}} )_{L^2,\Ad} = 0,\] which is equivalent to \eqref{eq:compatability for ymhs}.
\end{proof}

\subsection{Temporal gauge}
In this section we write $(x^0, x^1, x^2, x^3) = (t,x) \in \R^{1+3}$ for the Cartesian coordinates.
A connection $A \in \Omega^1(M;\mathfrak{g})$ is said to be in the {\it temporal gauge} if $A_0 = 0$ where $A = A_\alpha dx^\alpha$.

For a connection $V \in \Omega^1(\mathbb D; \mathfrak g)$ we define a connection $\mathscr T(V)$ in the temporal gauge by
    \begin{align}\label{temporal_U}
\mathscr T(V) = V \cdot \U,
\quad \text{where} \quad
\begin{cases}
\p_t \U = -V_0 \U,
\\
\U|_{t = \psi(x)} = \id,
\end{cases}
    \end{align}
and $\psi(x) = |x|-1$.
Observe that $\{(t,x) \in \mathbb D : t = \psi(x)\} = \p^- \mathbb D$ and $\U \in G^0(\mathbb D,p)$.
Occasionally we write also for $(V, \Psi) \in C^\infty(\mathbb D;T^*\mathbb D\otimes\mathfrak{g} \oplus \mathcal W)$
	\begin{align*}
\mathscr T(V, \Psi) = (V, \Psi) \cdot \U,
	\end{align*}
where $\U$ is as above. 

\begin{proposition}\label{prop_tempg_uniq}
Let $(V^{(j)}, \Psi^{(j)}) \in C^3(\mathbb D;T^*\mathbb D\otimes\mathfrak{g} \oplus \mathcal W)$, $j=1,2$, satisfy
\eqref{eq:ymhs1}--\eqref{eq:ymhs2}
in $\mathbb D$ with the same right-hand side $(J,\mathcal F)$,
and suppose that 
there is $\U \in G^0(\mathbb D, p)$ such that 
    \begin{align}\label{gauge_equiv_temp}
(V^{(1)}, \Psi^{(1)}) = (V^{(2)}, \Psi^{(2)}) \cdot \U 
    \end{align}
near $\p^- \mathbb D$.
Then we have the following uniqueness results:
\begin{itemize}
\item[(i)] If $(J, \mathcal F) = 0$ in $\mathbb D$ then 
$(V^{(1)}, \Psi^{(1)}) \sim (V^{(2)}, \Psi^{(2)})$
in $\mathbb D$.
\item[(ii)] If $(J, \mathcal F) = 0$
in $\mathbb D \setminus \mho$, $\U = \id$ in $\mho$ near $\p^- \mathbb D$, and if both $V^{(j)}$, $j=1,2$, are in the temporal gauge, then $\U$ does not depend on $t$, and \eqref{gauge_equiv_temp} holds in $\mathbb D$.
\end{itemize}
\end{proposition}

The proof of Proposition \ref{prop_tempg_uniq}
is very similar to that of \cite[Proposition 2]{CLOP2}.
However, for the convenience of the reader, we outline it. 
Let 
    \begin{align*}
(A, \Phi) \in C^3(\mathbb D;T^*\mathbb D\otimes\mathfrak{g} \oplus \mathcal W)
    \end{align*}
and suppose that $A$ is in the temporal gauge.
Write $D_{A}^*F_{A}+\J_{\rho}(d_{A}\Phi,\Phi) = J$.
Then 
    \begin{align*}
D_A^* F_A
&= \left(
\partial_\beta (\partial^\alpha A_\alpha)
-\partial^\alpha \partial_\alpha A_\beta
- [\partial^\alpha  A_\alpha, A_\beta]\right.
\\&\qquad \left.
- 2 [A^\alpha, \partial_\alpha A_\beta]
+ [A^\alpha, \partial_\beta A_\alpha]
- [A^\alpha, [A_\alpha, A_\beta]]\right) dx^\beta,
    \end{align*}
see e.g.   \cite[Lemma 12]{CLOP2}.
Taking $\beta = 0$ we get the constraint equation
    \begin{align}\label{constraint}
\p_0 (\partial^a A_a) + \tilde N_0 = J_0,
    \end{align}
with $a=1,2,3$, and taking $\beta = j = 1,2,3$ we get
    \begin{align}\label{YM_red}
\partial_j(\partial^a A_a)
-\partial^\alpha \partial_\alpha A_j
+ \tilde N_j = J_j.
    \end{align}
Here $\tilde N_\alpha$, $\alpha=0,1,2,3$, contain the terms that are of order one and zero,
    \begin{align*}
\tilde N_0 =
&
[A^a, \partial_0 A_a] + \J_\rho(\p_0 \Phi, \Phi),
\\
\tilde N_j =
&
-[\partial^a  A_a, A_j] - 2 [A^a, \partial_a A_j] + [A^a, \partial_j A_a] - [A^a, [A_a, A_j]] 
\\&\quad
+ \J_\rho\left(\p_j \Phi + \rho_*(A_j) \wedge \Phi, \Phi\right).
    \end{align*}
In the remainder of this section, we will use systematically Greek letters for indices over $0,1,2,3$ and Latin letters for $1,2,3$.

We differentiate (\ref{constraint}) using $\p_j$ and (\ref{YM_red}) using $\p_0$, and eliminate the common term $\p_j \p_0 (\partial^a A_a)$. This gives 
    \begin{align}\label{YM_red2}
\Box \p_t A_j + N_j = \p_t J_j - \p_j J_0,
    \end{align}
where $N_j =  -\p_j \tilde N_0 + \p_0 \tilde N_j$.
We write $d_{A}^*d_{A}\Phi +\mathcal{V}'(|\Phi|^{2})\Phi = \mathcal F$
and differentiate in time to obtain
    \begin{align}\label{H_red2}
\Box \p_t \Phi + N = \p_t \mathcal F
    \end{align}
where 
    \begin{align*}
N = -\p_t ( 2 \rho_*(A_a) \p^a \Phi - \p^a \rho_*(A_a) \Phi - \rho_*(A_a) \rho_*(A^a) \Phi) + \p_t (\mathcal V'(|\Phi|^2) \Phi).
    \end{align*}
We call (\ref{YM_red2})--(\ref{H_red2}) the reduced Yang--Mills--Higgs equations.

Recall that $\mathcal V'$ is affine.
Thus $N_j$ and $N$ are third order polynomials in $(A, \Phi)$.
Suppose now that $(\tilde A, \tilde\Phi)$ also
satisfies (\ref{YM_red2})--(\ref{H_red2}), and write
    \begin{align*}
u = (A - \tilde A, \Phi - \tilde \Phi), 
\quad
v = \p_t u.
    \end{align*}
By inspecting $N_j$ and $N$, we see that 
    \begin{align*}
\Box v + X_1 v + X_2 u = 0
    \end{align*}
where $X_j$, $j=1,2$, are first order differential operators in the $x^1, x^2$ and $x^3$ variables, with coefficients that depend on $(A, \Phi)$ and $(\tilde A, \tilde\Phi)$, and whence also on the $x^0$ variable.
The above equation together with $\p_t u - v= 0$ gives a system of the form studied in \cite[Section B.1, Appendix B]{CLOP2}.
In particular, if $(A, \Phi)$ and $(\tilde A, \tilde\Phi)$ coincide near $\p^- \mathbb D$ then they coincide in the whole $\mathbb D$ by \cite[Lemma 14]{CLOP2}.

\begin{proof}[Proof of Proposition \ref{prop_tempg_uniq}]
We begin with claim (i).
We define
	\begin{align*}
(\tilde V^{(j)}, \tilde \Psi^{(j)}) = \mathscr T(V^{(j)}, \Psi^{(j)}).
	\end{align*}
As $(V^{(1)}, \Psi^{(1)}) \sim (V^{(2)}, \Psi^{(2)})$ near $\p^- \mathbb D$ also
$(\tilde V^{(1)}, \tilde \Psi^{(1)}) \sim (\tilde V^{(2)}, \tilde \Psi^{(2)})$ there. That is, there is $\tilde \U \in G^0(\mathbb D,p)$ such that
    \begin{align*}
(\tilde V^{(1)}, \tilde \Psi^{(1)}) = (\tilde V^{(2)}, \tilde \Psi^{(2)}) \cdot \tilde \U, \quad \text{near $\p^- \mathbb D$}.
    \end{align*}
As both $\tilde V^{(1)}$ and $\tilde V^{(2)}$ are in the temporal gauge, $\tilde \U$ does not depend on time and we may define $(\hat V^{(1)}, \hat \Psi^{(1)}) = (\tilde V^{(2)}, \tilde \Psi^{(2)}) \cdot \tilde \U$ in the whole $\mathbb D$. Now both $(\hat V^{(1)}, \hat \Psi^{(1)})$ and $(\tilde V^{(1)}, \tilde \Psi^{(1)})$ satisfy the Yang--Mills--Higgs equations in $\mathbb D$. They are also both in the temporal gauge and coincide near $\p^- \mathbb D$. 
Passing to the reduced equations (\ref{YM_red2})--(\ref{H_red2})
and using \cite[Lemma 14]{CLOP2}, we see that they coincide in $\mathbb D$.
Therefore $(\tilde V^{(1)}, \tilde \Psi^{(1)}) \sim (\tilde V^{(2)}, \tilde \Psi^{(2)})$ in $\mathbb D$ and hence also $(V^{(1)}, \Psi^{(1)}) \sim (V^{(2)}, \Psi^{(2)})$ there.

Let us now turn to claim (ii). 
As $V^{(1)}$ and $V^{(2)}$ are in the temporal gauge, $\U$
does not depend on time and is well-defined and in all $\mathbb D$. Moreover, $\U = \id$ in $\mho$.
We define $(\tilde V^{(1)}, \tilde \Psi^{(1)}) =  (V^{(2)}, \Psi^{(2)})\cdot \U$. Then $(\tilde V^{(1)}, \tilde \Psi^{(1)}) = (V^{(1)}, \Psi^{(1)})$ near $\p^- \mathbb D$, and we proceed to show that the same holds in the whole of $\mathbb D$.

As $(V^{(2)}, \Psi^{(2)})$ satisfies the Yang--Mills--Higgs equations with vanishing right-hand side in $\mathbb D \setminus \mho$, 
the same holds for $(\tilde V^{(1)}, \tilde \Psi^{(1)})$, due to the gauge equivalence.
As $\U = \id$ in $\mho$, we see that
 $(\tilde V^{(1)}, \tilde \Psi^{(1)})$ satisfies the Yang--Mills--Higgs equations with the same right-hand side in $\mho$ as
$(V^{(j)}, \Psi^{(j)})$, $j=1,2$.
Moreover, $\tilde V^{(1)}$ is in the temporal gauge since $\U$ does not depend on $t$.
Hence $(\tilde V^{(1)}, \tilde \Psi^{(1)})$ and $(V^{(1)}, \Psi^{(1)})$ are two solutions
to the reduced Yang--Mills--Higgs equations (\ref{YM_red2})--(\ref{H_red2}), with the same right-hand side. It follows again from \cite[Lemma 14]{CLOP2} that $(\tilde V^{(1)}, \tilde \Psi^{(1)}) = (V^{(1)}, \Psi^{(1)})$ in $\mathbb D$.
\end{proof} 
 
\subsection{Relative Lorenz gauge}\label{sec_rel_lor}

We begin by writing \eqref{eq:ymhs1}--\eqref{eq:ymhs2} 
in terms of the perturbation fields 
    \begin{align*}
W = V - A, \quad \Upsilon = \Psi - \Phi.
    \end{align*}

\begin{lemma}\label{lem_ymh_pert}
Suppose that $(A, \Phi)$ satisfies \eqref{eq:ymh1}--\eqref{eq:ymh2}.
Then $(V, \Psi)$ satisfies \eqref{eq:ymhs1}--\eqref{eq:ymhs2}
if and only if the perturbation field $(W, \Upsilon)$
satisfies
	\begin{align}
		D_{A}^*D_{A}W  +\star[W,\star F_{A}]+\mathcal{N}_{A}(W)+\sum_{i=1}^{3}\M_{A,\Phi}^{i}(W,\Upsilon)&=J,
	\label{eq:perturbed YMH1 prelim} \\	
		\Box_{A,\rho} \Upsilon+ \sum_{i=1}^{3} 	\mathcal{N}_{A, \Phi}^i(W, \Upsilon)&=\mathcal{F},
	\label{eq:perturbed YMH2 prelim}
	\end{align}
where 
\begin{align*} \mathcal N_{A}(W) &= \frac{1}{2} D_A^\ast [W, W] + \star [W, \star D_A W] + \frac{1}{2} \star [W, \star [W, W]], \\
	\M_{A,\Phi}^{1}(W,\Upsilon)	&=\J_{\rho}(d_{A}\Upsilon,\Phi)+\J_{\rho}(d_{A}\Phi,\Upsilon)+\J_{\rho}(\rho_{*}(W)\Phi,\Phi),\\
	\M_{A,\Phi}^{2}(W,\Upsilon) 	&=\J_{\rho}(d_{A}\Upsilon,\Upsilon)+\J_{\rho}(\rho_{*}(W)\Upsilon,\Phi)+\J_{\rho}(\rho_{*}(W)\Phi,\Upsilon),\\
	\M_{A,\Phi}^{3}(W,\Upsilon) 	&=\J_{\rho}(\rho_{*}(W)\Upsilon,\Upsilon),\\
	\mathcal{N}_{A, \Phi}^1(W, \Upsilon) &= 
d_{A}^*(\rho_{*}(W)\Phi)+\star(\rho_{*}(W)\wedge\star d_{A}\Phi)
+ 2 \Re \langle \Phi, \Upsilon \rangle_{\mathcal W} \Phi + \mathcal{V}'(|\Phi|^2) \Upsilon,
\\
	\mathcal{N}_{A, \Phi}^2(W, \Upsilon) &= \star(\rho_{*}(W)\wedge\star \rho_{*}(W)\Phi) 
+ 
d_{A}^*(\rho_{*}(W)\Upsilon)+\star(\rho_{*}(W)\wedge\star d_{A}\Upsilon)
\\&\qquad + 2 \Re \langle \Phi, \Upsilon \rangle_{\mathcal W} \Upsilon + |\Upsilon|^2 \Phi,
\\
	\mathcal{N}_{A, \Phi}^3(W, \Upsilon) &= \star(\rho_{*}(W)\wedge\star \rho_{*}(W)\Upsilon) 
+ |\Upsilon|^2 \Upsilon.
\end{align*}
\end{lemma}
\begin{proof}
Let us first consider the Yang--Mills channel \eqref{eq:ymhs1}. On the one hand, we know, see \cite[(20)--(21)]{CLOP2}, that
 \[D_{V}^*F_{V}=D_{A}^*F_{A}+D_{A}^*D_{A}W+\star [W,\star F_{A}]+\mathcal N_{A}(W).\]
 On the other hand, it follows from the bilinearity of $\J_{\rho}$ that
 $$\J_{\rho}(d_{V}\Psi,\Psi)   =\J_{\rho}(d_{A}\Phi,\Phi)+\M_{A,\Phi}^{1}(W,\Upsilon)+\M_{A,\Phi}^{2}(W,\Upsilon)+\M_{A,\Phi}^{3}(W,\Upsilon),$$
 Combining these calculations allows us to transform \eqref{eq:ymhs1} to \eqref{eq:perturbed YMH1 prelim}.

Next we turn to the Higgs channel \eqref{eq:ymhs2}. We recall that \begin{align*}d_{V}\Psi&=d_{A}\Psi+\rho_{*}(W) \wedge \Psi, \\ d_{V}^*(\cdot)&=d_{A}^*(\cdot)+\star (\rho_{*}(W)\wedge\star(\cdot) ).\end{align*}
Hence \[d_{V}^*d_{V}\Psi=
d_{A}^*d_{A}\Psi
+d_{A}^*(\rho_{*}(W)\Psi)+\star(\rho_{*}(W)\wedge\star d_{A}\Psi)
+\star(\rho_{*}(W)\wedge\star \rho_{*}(W)\Psi).\]
Substitution $\Psi = \Upsilon + \Phi$
gives the terms containing $W$ in $\mathcal N_{A,\Phi}^i$, $i=1,2,3$.
To finish the proof, we recall that $\mathcal{V}(s) = s^{2}/2-s$. 
Thus 
    \begin{align*}
\mathcal V'(|\Psi|^2) \Psi 
- 
\mathcal V'(|\Phi|^2) \Phi 
= -\Upsilon + |\Psi|^2 \Psi - |\Phi|^2 \Phi.
    \end{align*}
Plugging in $\Psi = \Upsilon + \Phi$ and expanding this cubic polynomial gives the rest of the terms in $\mathcal N_{A,\Phi}^i$, $i=1,2,3$. 
\end{proof}
  
Recall that the relative Lorenz gauge condition reads $D_A^* W = 0$.
In the relative Lorenz gauge $D_A^* D_A W = \Box_{A,\Ad} W$,
and (\ref{eq:perturbed YMH1 prelim})--(\ref{eq:perturbed YMH2 prelim})
is a system of semilinear wave equations.
Together with suitable initial conditions, it has a unique solution when the source $(J, \mathcal F)$ is small and smooth enough.
However, if $(W, \Upsilon)$ is its solution then $(V,\Psi) = (W, \Upsilon) + (A, \Phi)$ solves the perturbed Yang--Mills--Higgs equations \eqref{eq:ymhs1}--\eqref{eq:ymhs2} if and only if $D_A D_A^\ast W=0$. This again follows from the compatibility condition in Lemma \ref{lem_comp_cond}.

\begin{lemma}\label{lem_undo_gauge}
Suppose that $(A, \Phi)$ satisfies \eqref{eq:ymh1}--\eqref{eq:ymh2}
and that $(W,\Upsilon)$ satisfies 
	\begin{align}
		\Box_{A,\Ad} W+\star[W,\star F_{A}]+\mathcal{N}_{A}(W)+\sum_{i=1}^{3}\M_{A,\Phi}^{i}(W,\Upsilon)&=J,
	\label{eq:perturbed YMH1} \\	
		\Box_{A,\rho} \Upsilon+ \sum_{i=1}^{3} 	\mathcal{N}_{A, \Phi}^i(W, \Upsilon)&=\mathcal{F},
	\label{eq:perturbed YMH2}
	\end{align}
together with vanishing initial conditions.
Suppose, furthermore, that 
    \begin{align}\label{comp_equiv}
D^*_{A} J + \star [W, \star J] = \J_\rho(\mathcal F, \Upsilon) + \J_\rho(\mathcal F, \Phi).            
    \end{align}
Then $W$ satisfies the relative Lorenz gauge condition 
$D_A^* W = 0$, and 
$V = A + W$ and $\Psi = \Phi + \Upsilon$ satisfy \eqref{eq:ymhs1}--\eqref{eq:ymhs2}.
\end{lemma} 
\begin{proof}
Let $Q \in C_0^\infty(M, \mathfrak g)$. 
The definitions of $D_V^\ast$ and $\J_\rho$ yield that
\begin{align*}
	&(D^*_{V}\J_{\rho}(d_{V}\Psi,\Psi),Q)_{L^2, \Ad}=(\J_{\rho}(d_{V}\Psi,\Psi),D_{V}Q)_{L^2, \Ad} =\Re (d_{V}\Psi,\rho_{*}(D_{V}Q)\Psi)_{L^2, E}  
\\&\qquad
=
\Re(d_{V}\Psi,d_{V}(\rho_{*}(Q)\Psi))_{L^2, E}-\Re(d_{V}\Psi,\rho_{*}(Q)(d_{V}\Psi))_{L^2, E}
\\&\qquad
=\Re(d_{V}^*d_{V}\Psi,\rho_{*}(Q)\Psi)_{L^2, E}-\Re(d_{V}\Psi,\rho_{*}(Q)(d_{V}\Psi))_{L^2, E}.
\end{align*}
By Lemma \ref{lem_ymh_pert}, equations (\ref{eq:perturbed YMH2}) 
and (\ref{eq:ymhs2}) are equivalent.
Invoking \eqref{eq:ymhs2} and the $G$-invariance of the metric on $W$, cf. Remark \ref{rem_antisymmetry}, in turn leads to
\begin{align*}
	(D^*_{V}\J_{\rho}(d_{V}\Psi,\Psi),Q)_{L^2, \Ad}
	&=\Re(\mathcal{F} - \mathcal{V}'(|\Psi|^{2})\Psi,\rho_{*}(Q)\Psi)_{L^2, E}-\Re(d_{V}\Psi,\rho_{*}(Q)(d_{V}\Psi))_{L^2, E}\\
	&=\Re(\mathcal{F},\rho_{*}(Q)\Psi)_{L^2, E}
=(\J_{\rho}(\mathcal{F},\Psi),Q)_{L^{2},\text{Ad}}.
\end{align*} 
Observe that (\ref{comp_equiv}) is equivalent with (\ref{eq:compatability for ymhs}).
Hence 
    \begin{align}\label{psi_compatibility}
D_V^* \J_\rho(d_V \Psi, \Psi) = \J_\rho(\mathcal F, \Psi) = D_{V}^*J.
    \end{align}

By Lemma \ref{lem_ymh_pert}, 
equations (\ref{eq:perturbed YMH1 prelim}) 
and (\ref{eq:ymhs1}) are equivalent. 
On the other hand, the left-hand side of (\ref{eq:perturbed YMH1 prelim}) and the left-hand side of (\ref{eq:perturbed YMH1}) differ by $D_A D_A^\ast W$.
Therefore, 
    \begin{align*}
D_A D_A^\ast W + D_{V}^*F_{V}+\J_{\rho}(d_{V}\Psi,\Psi)=J.
    \end{align*}
Using (\ref{psi_compatibility}) and $D_V^\ast D_V^\ast F_V = 0$, see e.g. \cite[Lemma 2]{CLOP2} for the latter fact,
we see that 
    \begin{align*}
D_V^* D_A D_A^\ast W = 0.
    \end{align*}
But this can be viewed as a wave equation for $D_A^\ast W$. As $W$ vanishes initially, it follows that $D_A^\ast W$ vanishes identically. Thus (\ref{eq:perturbed YMH1}) reduces to (\ref{eq:perturbed YMH1 prelim}), and $(V,\Psi)$ satisfies \eqref{eq:ymhs1}--\eqref{eq:ymhs2} by Lemma \ref{lem_ymh_pert}.
\end{proof} 

\begin{remark}
If $D_A^* W = 0$ then the expressions $\mathcal{N}_{A, \Phi}^j$, $j=1,2$,
simplify as follows
    \begin{align*}
	\mathcal{N}_{A, \Phi}^1(W, \Upsilon) &= 
2\star(\rho_{*}(W)\wedge\star d_{A}\Phi)
+ 2 \Re \langle \Phi, \Upsilon \rangle_{\mathcal W} \Phi + (|\Phi|^2 -1)\Upsilon,
\\
	\mathcal{N}_{A, \Phi}^2(W, \Upsilon) &= \star(\rho_{*}(W)\wedge\star \rho_{*}(W)\Phi) 
+ 
2\star(\rho_{*}(W)\wedge\star d_{A}\Upsilon)
\\&\qquad + 2 \Re \langle \Phi, \Upsilon \rangle_{\mathcal W} \Upsilon + |\Upsilon|^2 \Phi.
    \end{align*}
Likewise $\mathcal N_A$ simplifies slightly, cf. (29) in \cite{CLOP2}.
Lemma \ref{lem_undo_gauge} holds if $\mathcal{N}_{A, \Phi}^j$, $j=1,2$,
and $\mathcal N_A$ are replaced by the simplifications. 
\end{remark}

We view \eqref{comp_equiv} as an equation for $J_0$,
and consider only the spatial part 
	\begin{align}\label{def_spatial_part}
J - J_0 dx^0
	\end{align}
of $J$ as a source in \eqref{eq:perturbed YMH1}.
Now we are ready to prove that the perturbed Yang--Mills--Higgs system in the relative Lorenz gauge \eqref{eq:perturbed YMH1}--\eqref{eq:perturbed YMH2}, coupled with the compatibility condition \eqref{comp_equiv} and vanishing initial conditions, has a unique solution for a small source.

\begin{proposition}\label{prop:forward problem relative}
Write $M_1 = (-1,1) \times \R^3$ and let $k \geq 4$.
	Suppose the Yang--Mills--Higgs field $(A, \Phi) \in \Omega^1(M_1, \g) \times C^\infty(M_1, \mathcal{W})$ is bounded and has bounded derivatives. Then there is a neighbourhood $\mathcal{H}$ of the zero function in 
	\[H^{k + 2}(M_1, \g \oplus \g \oplus \g \oplus \W)\]
such that for all $(J_1, J_2, J_3, \mathcal{F}) \in \mathcal{H}$ the system \eqref{eq:perturbed YMH1}--\eqref{comp_equiv} has a unique solution $$(W, \Upsilon, J_0) \in H^{k + 1}(M_1, T^\ast M_1 \otimes \g \oplus \mathcal{W} \oplus \g)$$ 
vanishing for $t<-1$, and the map from $(J_1, J_2, J_3, \mathcal{F})$ to $(W, \Upsilon, J_0)$ is smooth.

Furthermore, the solutions have the following finite speed of propagation property. 
For $x_0 \in M$ and $\epsilon > 0$, define
	\begin{align*}
\mathbb{D}_{(x_0, \epsilon)} = \{\epsilon y + x_0 : y \in \mathbb{D}\}, \quad
\p^- \mathbb{D}_{(x_0, \epsilon)} = \{\epsilon y + x_0 : y \in \p^- \mathbb{D}\}.
	\end{align*}
If $(W, \Upsilon, J_0)$ solves \eqref{eq:perturbed YMH1}--\eqref{comp_equiv} with $(J_1, J_2, J_3, \mathcal F) = 0$
and if it vanishes near $\p^- \mathbb{D}_{(x_0, \epsilon)}$
then it vanishes on the whole of $\mathbb{D}_{(x_0, \epsilon)}$.
\end{proposition}

\begin{proof}
This is proved in the same spirit as \cite[Proposition 3]{CLOP2}. We sketch the major steps.
Note that (\ref{comp_equiv}) can be viewed as the following ordinary differential equation for $J_0$,
    \begin{align}\label{eq_J0}
\p_t J_0 + [A_0, J_0] + [W_0, J_0] 
&= 
\p^j J_j + [A^j, J_j] + [W^j, J_j]
-\J_\rho(\mathcal F, \Upsilon) - \J_\rho(\mathcal F, \Phi),
    \end{align}
where $j=1,2,3$.
Similarly to Proposition \ref{prop_tempg_uniq}, we can show that the system (\ref{eq:perturbed YMH1})--(\ref{comp_equiv}) has at most one solution. Suppose that both $(W^{(j)}, \Upsilon^{(j)}, J_0^{(j)})$, $j=1,2$,
solve the system with the same source $(J_1, J_2, J_3, \mathcal F)$.
An inspection of the lower order terms shows that $(v,u)$, with
    \begin{align*}
v = (W^{(1)} - W^{(2)}, \Upsilon^{(1)} - \Upsilon^{(2)}),
\quad
u = J_0^{(1)} - J_0^{(2)},
    \end{align*}
satisfies a system of the form
    \begin{align*}
\Box v + X_1 v + X_2 u &= 0,
\\
\p_t u + Y_1 v + Y_2 u &= 0,
    \end{align*}
where $X_j$ and $Y_j$ are differential operators of order one and zero, respectively. The coefficients of $X_j$ and $Y_j$ depend on $(W^{(j)}, \Upsilon^{(j)}, J_0^{(j)})$, $j=1,2$, and also on the source $(J_1, J_2, J_3, \mathcal F)$. It is also important for us that $X_2 = -1$ is of zeroth order with respect to the time variable.   \cite[Lemma 14]{CLOP2} implies now that 
if $(W^{(j)}, \Upsilon^{(j)}, J_0^{(j)})$, $j=1,2$, vanish initially  then they coincide everywhere. 
In fact, \cite[Lemma 14]{CLOP2} gives right away also the stronger finite speed of propagation result concerning $\mathbb{D}_{(x_0, \epsilon)}$.

Let us now turn to existence of a solution.
We define two operators $\mathcal{P}$ and $\mathcal{K}$ to separate the linear terms from the nonlinear terms and source terms in \eqref{eq:perturbed YMH1}--\eqref{eq:perturbed YMH2},
 	\[\mathcal{P}(W, \Upsilon, J_0) = \left(\begin{array}{l}
	   \Box_{A, \Ad}W+\star[W,\star F_{A}] + \M_{A,\Phi}^{1}(W,\Upsilon) - J_0dx^0  \\	
	 	\Box_{A, \rho}\Upsilon + \mathcal{N}_{A,\Phi}^{1}(W,\Upsilon) \\
 	   \partial_0 J_0 + [A_0, J_0] 
\end{array}\right),\]
  	\[\mathcal{K}(W, \Upsilon, J, \mathcal{F}) =\left(\begin{array}{l}
  J_jdx^j-\mathcal{N}_{A}(W)-\sum_{i=2}^{3}\M_{A,\Phi}^{i}(W,\Upsilon)  \\	
  \mathcal{F}-\sum_{i=2}^{3}\mathcal{N}_{A,\Phi}^{i}(W,\Upsilon)\\
\p^j J_j + [A^j, J_j] + [W^j, J_j]
-\J_\rho(\mathcal F, \Upsilon + \Phi) 
- [W_0, J_0]
 \end{array}\right).\]

Writing $u = (W,\Upsilon, J_0)$, we consider the linear system  $\mathcal Pu = f$ with vanishing initial conditions at $t=-1$. 
As the equation for $J_0$ is decoupled from the rest of the system, it is easy to solve. The system induces a continuous map 
$\mathcal S f = u$, 
 \begin{align*} 
\mathcal S &: H^{k}(M_1; T^\ast M_1 \otimes \g) \times H^{k}(M_1;  \mathcal{W})  \times H^{k+1}(M_1; \g)  \\&\qquad \rightarrow  H^{k + 1}(M_1; T^\ast M \otimes \g\oplus\mathcal{W}\oplus\g).  \end{align*}

Now one can argue that the map $$\mathcal{T}(W, \Upsilon, J, \mathcal{F}) = (W, \Upsilon, J_0) - \mathcal{S} \mathcal{K}(W, \Upsilon, J, \mathcal{F})$$ has the partial derivative in $(W, \Upsilon, J_0)$ at $0$ equal to $\id$. Then the claim follows from the implicit function theorem. 
 \end{proof}

 \subsection{The source-to-solution map}
 
In this section we show how Remark \ref{rem_ss_map} can be made precise. 
 
   \begin{lemma}\label{lem_pass_temp_gauge}
Let $(A, \Phi)$ be a Yang--Mills--Higgs field in $\mathbb{D}$.
Then there is $(\tilde A', \tilde \Phi')$ in $\mathcal D_{(A,\Phi)}$
such that $\tilde A'$ is in the temporal gauge in $\mho$ 
and that \eqref{eq:ymh1}--\eqref{eq:ymh2} hold in $\mho$.
Moreover, for any such $(\tilde A', \tilde \Phi')$ there is $(\tilde{A}, \tilde{\Phi})$, a Yang--Mills--Higgs field in the temporal gauge in $\mathbb{D}$, 
satisfying $(\tilde{A}, \tilde{\Phi}) \sim (A, \Phi)$ 
and $(\tilde{A}, \tilde{\Phi})|_\mho = (\tilde A', \tilde \Phi')$.
 	\end{lemma} 

Before giving the short proof, let us explain how the lemma is used. As $\mathcal D_{(A,\Phi)}$ can determine 
$(A, \Phi)$ only up to a gauge and as $\mathcal D_{(A,\Phi)} = \mathcal D_{(\tilde A,\tilde \Phi)}$, by replacing $(A,\Phi)$ with $(\tilde A,\tilde \Phi)$, we may assume without loss of generality that $A$ is in the temporal gauge and that $(A, \Phi)|_\mho$ is ``known'' in the sense that we have fixed a gauge equivalent copy of $(A,\Phi)$ in $\mho$. 

This replacement should not cause confusion, however, some care is needed when interpreting the results below. We will show that $\mathcal D_{(A,\Phi)}$ determines the broken light ray transform $\mathbf{S}_{z \gets y \gets x}^{A, \Ad \oplus \rho}$
and apply Corollary \ref{corollary:rhofaithful}
to show that $(A,\Phi)$ is then determined up to the action of a gauge transformation $\U$.
Although this gauge transformation satisfies $\U = \id$ in $\mho$,
 it does not follow that the gauge invariance is eliminated in $\mho$.
Rather, in $\mho$, the construction simply yields the gauge equivalent copy of $(A, \Phi)$ that we happened to fix. 
 	
\begin{proof}
To show the existence of $(\tilde{A}', \tilde{\Phi}')$,
we simply let $(\tilde{A}', \tilde{\Phi}') = \mathscr T(A, \Phi)|_\mho$.
Suppose now that $(\tilde A', \tilde \Phi') \in \mathcal D_{(A,\Phi)}$, $\tilde A'$ is in the temporal gauge in $\mho$ 
and that \eqref{eq:ymh1}--\eqref{eq:ymh2} hold in $\mho$.
Now the definition of $\mathcal D_{(A,\Phi)}$ implies that there is 
such a Yang--Mills--Higgs field $(B,\Xi)$ in $\mathbb D$ that $(\tilde A', \tilde \Phi') = (B,\Xi)|_\mho$ and that $(A, \Phi) \sim (B,\Xi)$ near $\p^- \mathbb D$. Let $(\tilde A, \tilde \Phi) = \mathscr T(B,\Xi)$. It follows from $\tilde A'$ being in the temporal gauge and $\tilde{A}' = B|_\mho$ that $(\tilde A', \tilde \Phi') = (\tilde A, \tilde \Phi)|_\mho$. Moreover, $(\tilde A, \tilde \Phi) \sim (A, \Phi)$ near $\p^- \mathbb D$. Finally, Proposition \ref{prop_tempg_uniq} implies that $(\tilde A, \tilde \Phi) \sim (A, \Phi)$ in the whole of $\mathbb D$.
\end{proof}
 	 
\begin{proposition} \label{prop : source-to-solution map}  
For any point $x_0 \in \mho$ there exist a neighbourhood $\mho_0 \subset \mho$ of $x_0$ and a neighbourhood $\mathcal{H}_0$ of the zero function in $H_0^7(\mho_0; \g \oplus \g \oplus \g \oplus \mathcal{W})$ such that the data set $\mathcal{D}_{(A, \Phi)}$ determines the following source-to-solution map $\mathbf{L}$ on $\mathcal H_0$, 
\begin{align*}\mathbf{L}(J_1, J_2, J_3, \mathcal{F}) =\mathscr{T}(V, \Psi)|_\mho  ,  \end{align*} 
where $V = W + A$, $\Psi = \Upsilon + \Phi$ and $(W, \Upsilon, J_0)$ solves the system \eqref{eq:perturbed YMH1}--\eqref{comp_equiv} for given $(J_1, J_2, J_3, \mathcal{F}) \in \mathcal{H}_0$.
\end{proposition}
\begin{proof}
We fix a point $x_0 = (t_0, x_0^1, x_0^2, x_0^3) \in \mho$ and then choose $\epsilon > 0$ small enough so that the diamond-like neighbourhood $\mathbb{D}_{(x_0, \epsilon)}$ of $x_0$
satisfies $\mathbb{D}_{(x_0, \epsilon)} \subset \mho$.
We denote by $\mho_0$ the interior of 
	\begin{align*}
\{ (t, x^1, x^2, x^3) \in \mathbb{D}_{(x_0, \epsilon)} : t < t_0 \}.
	\end{align*}

Let $\mathcal H$ be the neighbourhood of the zero function for $k = 5$ in Proposition \ref{prop:forward problem relative}
and let 
	\begin{align}\label{H7_sources}
(J_1, J_2, J_3, \mathcal{F}) \in H_0^7(\mho_0; \g\oplus\g\oplus\g\oplus\mathcal{W}) \cap \mathcal H.
	\end{align}
Write $(W, \Upsilon, J_0)$ for the solution of \eqref{eq:perturbed YMH1}--\eqref{comp_equiv} in $M_1$.
(Due to the finite speed of propagation property in Proposition \ref{prop:forward problem relative}, the fact that $(A,\Phi)$ is defined only on $\mathbb D$ is inconsequential.)
It follows from the finite speed of propagation that 
$(W, \Upsilon, J_0)$ vanishes near $\p^- \mathbb D_{(x_0, \epsilon)}$ and in
	\begin{align*}
\{ (t, x^1, x^2, x^3) \in \mathbb D \setminus \mho_0 : t < t_0 \}.
	\end{align*}
By Lemma \ref{lem_pass_temp_gauge} we may assume that $(A, \Phi)$
is known in $\mho$, and the finite speed of propagation implies then that 
$(W, \Upsilon, J_0)$ can be computed in $\mho_0$.

Writing
	\begin{align*}
\mho_- = \{ (t, x^1, x^2, x^3) \in \mho : t < t_0 \},
\quad
\mho_+ = \{ (t, x^1, x^2, x^3) \in \mho : t > t_0 \},
	\end{align*}
we will show that the source-to-solution map satisfies 
	\begin{align}\label{L_characterization}
\mathbf{L}(J_1, J_2, J_3, \mathcal{F}) =   \mathscr{T}(V', \Psi'),
	\end{align}
where $(V', \Psi') \in  \mathcal{D}_{(A, \Phi)}$ is such that
\begin{itemize}
\item[(L1)]  $(V', \Psi')$ coincides with $(V, \Psi)$  in $\mho_-$, 
\item[(L2)] the spatial component of $D^\ast_{V'} F_{V'} + \J_\rho(d_{V'}\Psi', \Psi')$  vanishes in $\mho_+$,
\item[(L3)] $d^\ast_{V'} d_{V'} \Psi' + \mathcal{V}'(|\Psi'|^2) \Psi'$ vanishes in $\mho_+$. 
\end{itemize}
Observe that $(V,\Psi)|_\mho$
satisfies (L1)--(L3). The smoothness assumption (\ref{H7_sources}) together with Proposition \ref{prop:forward problem relative} and the Sobolev embedding $H^6(M_1) \subset C^3(M_1)$ guarantee that $(V,\Psi) \in \mathcal{D}_{(A, \Phi)}$.  
Observe, furthermore, that (L2) and (L3) can be verified given any $(V', \Psi') \in \mathcal{D}_{(A, \Phi)}$. The argument above shows that 
$(V, \Psi)$ can be determined in $\mho_-$, and so (L1) can also be verified by using the data. 
Thus it remains to show \eqref{L_characterization}.

Let $(V', \Psi') \in \mathcal{D}_{(A, \Phi)}$ satisfy (L1)--(L3).
The definition of $\mathcal{D}_{(A, \Phi)}$ implies that 
there are $(\tilde V, \tilde \Psi)$ and $\mathbf u$ such that
$(\tilde V, \tilde \Psi) = (A, \Phi) \cdot \mathbf u$ near $\p^- \mathbb D$.
In $\mho_-$ near $\p^- \mathbb D$
	\begin{align*}
(A, \Phi) \cdot \mathbf u = (\tilde V, \tilde \Psi) = (V', \Psi') = (V, \Psi) = (A, \Phi),
	\end{align*}
and it follows that $\mathbf u = \id$ there.
This is a consequence of the use of the pointed gauge group $G^0(\mathbb{D}, p)$, instead of the full gauge group $C^\infty(\mathbb{D}; G)$, for gauge equivalence. 
Indeed, due to $\mathbf u \in G^0(\mathbb{D}, p)$, there holds $\mathbf u(p) = \id$, and $A \cdot \mathbf u = A$ is equivalent with the differential equation $d\mathbf u = [\mathbf u, A]$. Integrating this equation along a curve starting from $p$ gives $\mathbf u = \id$ in $\mho$ near $\p^- \mathbb D$. 

Let us now pass to the temporal gauge. 
Let $\U$ be as in \eqref{temporal_U} so that $\mathscr T(V) = V \cdot \U$, and define $\tilde \U$ analogously so that 
$\mathscr T(\tilde V) = \tilde V \cdot \tilde \U$.
As $A$ is in the temporal gauge and as $\tilde V = V = A$ near $\p^- \mathbb D$, there holds $\tilde \U = \U = \id$ near $\p^- \mathbb D$.
In particular,
	\begin{align*}
\mathscr T(\tilde V, \tilde \Psi) 
= (\tilde V, \tilde \Psi) \cdot \tilde \U
= (A, \Phi) \cdot (\mathbf u \tilde \U)
= (V, \Psi) \cdot (\mathbf u \tilde \U)
= \mathscr T(V, \Psi) \cdot (\U^{-1}\mathbf u \tilde \U)
	\end{align*}
and $\U^{-1}\mathbf u \tilde \U = \id$ in $\mho$ near $\p^- \mathbb D$.
We will apply Proposition \ref{prop_tempg_uniq}
to $\mathscr T(\tilde V, \tilde \Psi)$ and $\mathscr T(V, \Psi)$.
This shows that $\U^{-1}\mathbf u \tilde \U$ does not depend on time and hence $\U^{-1}\mathbf u \tilde \U = \id$ in $\mho$.
Then \eqref{L_characterization} follows immediately,
and it remains to show that the assumptions of Proposition \ref{prop_tempg_uniq} hold. That is, we still need to show that 
$\mathscr T(\tilde V, \tilde \Psi)$ and $\mathscr T(V, \Psi)$
satisfy \eqref{eq:ymhs1}--\eqref{eq:ymhs2}
in $\mathbb D$ with the same right-hand side.

We write
	\begin{align*}
(V^{(1)}, \Psi^{(1)}) = \mathscr T(V, \Psi),
\quad
(V^{(2)}, \Psi^{(2)}) = \mathscr T(\tilde V,\tilde \Psi),
	\end{align*}
and denote by
\begin{align*}
J^{(j)} &= D_{V^{(j)}}^*F_{V^{(j)}}+\J_{\rho}(d_{V^{(j)}}\Psi^{(j)},\Psi^{(j)}),\\
\mathcal{F}^{(j)} &= d_{V^{(j)}}^*d_{V^{(j)}}\Psi^{(j)} +\mathcal{V}'(|\Psi^{(j)}|^{2})\Psi^{(j)},
 \end{align*} 
the right-hand side of \eqref{eq:ymhs1}--\eqref{eq:ymhs2}. 
There holds $(J^{(j)}, \mathcal F^{(j)}) = 0$, $j=1,2$, in $\mathbb D \setminus \mho$,
and $(J^{(1)}, \mathcal F^{(1)}) = (J^{(2)}, \mathcal F^{(2)})$
in $\mho_-$ since $(V^{(1)}, \Psi^{(1)}) = (V^{(2)}, \Psi^{(2)})$ there due to (L1). 
Let us now restrict our attention to $\mho_+$.
It follows from (L2) and (L3) that $\mathcal F^{(j)} = 0$
and $J^{(j)} = J^{(j)}_0 dx^0$ for some $J^{(j)}_0$.
Lemma \ref{lem_comp_cond} implies then that $D_{V^{(j)}}^* J^{(j)} = 0$. This again is equivalent with 
	\begin{align*}
\p_t J_0^{(j)} + [V_0^{(j)}, J_0] = 0,
	\end{align*}
cf. \eqref{eq_J0}.
But $V^{(j)}$ is in the temporal gauge and thus $\p_t J^{(j)}_0 = 0$.
Recalling that $J_0^{(1)} = J_0^{(2)}$ in $\mho_-$,
we see that $J_0^{(1)} = J_0^{(2)}$ in $\mho_+$ as well.
\end{proof}

 \section{Linearization of the perturbed Yang--Mills--Higgs system} \label{sec:linearization}
 
To recover the background Yang--Mills--Higgs field we will use a modification of the Kurylev--Lassas--Uhlmann multiple linearization scheme. In this section, we shall calculate the threefold linearization of the perturbed Yang--Mills--Higgs system \eqref{eq:perturbed YMH1}--\eqref{eq:perturbed YMH2} in the relative Lorenz gauge. To simplify notations, we shall omit the subscript $\mathcal{W}$ of $\langle \cdot, \cdot \rangle_\mathcal{W}$ hereafter.

  \subsection{Threefold linearization}
  
Given a Yang--Mills--Higgs field $(A, \Phi)$, we consider \eqref{eq:perturbed YMH1}--\eqref{eq:perturbed YMH2} with $J = \sum_{k=1}^3\epsilon_{(k)} J_{(k)}$ and $\mathcal{F} = \sum_{k=1}^3\epsilon_{(k)} \mathcal{F}_{(k)}$ for small $\epsilon_{(k)} \in \R$ and sections $J_{(k)}$ and $\mathcal{F}_{(k)}$ of $T^* M \otimes \g$ and $E$, respectively.
  For the threefold linearization, we consider the following derivatives of the perturbation field $(W, \Upsilon)$ in $\epsilon = (\epsilon_{(1)}, \epsilon_{(2)}, \epsilon_{(3)})$ at $\epsilon = 0$,
  \begin{eqnarray}\label{def_Y} \displaystyle
  	W_{(k)} = \frac{\partial W}{\partial \epsilon_{(k)}}\bigg|_{\epsilon = 0},
  	& \displaystyle W_{(kl)} = \frac{\partial^2 W}{\partial \epsilon_{(k)}\partial \epsilon_{(l)}}\bigg|_{\epsilon = 0},&
  	W_{(123)} = \frac{\partial^3 W}{\partial \epsilon_{(1)}\partial \epsilon_{(2)}\partial \epsilon_{(3)}}\bigg|_{\epsilon = 0};
  	\\ \label{def_Xi} \displaystyle
  	\Upsilon_{(k)} = \frac{\partial \Upsilon}{\partial \epsilon_{(k)}}\bigg|_{\epsilon = 0},
  	&
  	\displaystyle \Upsilon_{(kl)} = \frac{\partial^2 \Upsilon}{\partial \epsilon_{(k)}\partial \epsilon_{(l)}}\bigg|_{\epsilon = 0},
  	&
  	\Upsilon_{(123)} = \frac{\partial^3 \Upsilon}{\partial \epsilon_{(1)}\partial \epsilon_{(2)}\partial \epsilon_{(3)}}\bigg|_{\epsilon = 0}.    \end{eqnarray}
  
  Differentiating \eqref{eq:perturbed YMH1}--\eqref{eq:perturbed YMH2} results in the following system of linear wave equations for $(W_{(k)}, \Upsilon_{(k)})$
  \begin{align}
\label{eqn : linearized 1-1}
  	\Box_{A,\text{Ad}}W_{(k)} + \J_{\rho}(d_{A}\Upsilon_{(k)},\Phi) + Z_{(k)} & = J_{(k)},
\\\label{eqn : linearized 1-2}
  	\Box_{A,\rho}\Upsilon_{(k)} + \mathcal{Z}_{(k)} & = \mathcal{F}_{(k)},
  \end{align}  where the zeroth order terms take the form
  \begin{eqnarray*}
  	Z_{(k)} &=& \star[W_{(k)},\star F_{A}]+\J_{\rho}(d_{A}\Phi,\Upsilon_{(k)})+\J_{\rho}(\rho_{*}(W_{(k)})\Phi,\Phi),\\
  	\mathcal{Z}_{(k)} &=& 2\star(\rho_{*}(W_{(k)})\wedge\star d_{A}\Phi) + \mathcal{V}' (|\Phi|^2)    \Upsilon_{(k)}     + 2  \langle \Phi,  \Upsilon_{(k)} \rangle \Phi.
  \end{eqnarray*}
It turns out that zeroth order terms that are linear in the variables  (\ref{def_Y})--(\ref{def_Xi}) play no role in our subsequent computations. For this reason, we will not give the explicit form of these terms in the computations below.

 Differentiating \eqref{eq:perturbed YMH1}--\eqref{eq:perturbed YMH2}  twice, we obtain a system of linear wave equations for $(W_{(kl)},\Upsilon_{(kl)})$, 
    \begin{align}\label{eqn : linearized 2-1} 
  	\Box_{A,\text{Ad}}W_{(kl)} +\J_{\rho}(d_{A}\Upsilon_{(kl)},\Phi) +Z_{(kl)} &=N_{(kl)}, 
\\\label{eqn : linearized 2-2}
	\Box_{A,\rho}\Upsilon_{(kl)}+\mathcal{Z}_{(kl)}&=\mathcal{N}_{(kl)},
    \end{align}
  where $Z_{(kl)}$ and $\mathcal{Z}_{(kl)}$ are linear and zeroth order in $(W_{(kl)},\Upsilon_{(kl)})$,
and $N_{(kl)}$ and $\mathcal{N}_{(kl)}$ contain the nonlinear terms.
It turns out that the quadratic terms with no extra derivative play no role in the computations below. For this reason, we write 
    \begin{align}\label{NR_kl}
N_{(kl)} = \tilde N_{(kl)} + R_{(kl)}, 
\quad
\mathcal N_{(kl)} = \tilde{\mathcal N}_{(kl)} + \mathcal R_{(kl)}, 
    \end{align}
where the leading terms are
    \begin{align*}
\tilde{N}_{(kl)} 
&= 
-\frac{1}{2} d^\ast [W_{(k)}, W_{(l)}] - \star[W_{(k)}, \star d W_{(l)}] - \frac{1}{2} d^\ast [W_{(l)}, W_{(k)}] - \star[W_{(l)}, \star d W_{(k)}]  
\\&\qquad
- \J_{\rho}(d\Upsilon_{(k)},\Upsilon_{(l)})  - \J_{\rho}(d\Upsilon_{(l)},\Upsilon_{(k)}), 
\\
\tilde{\mathcal{N}}_{(kl)} 
&=  
- 2\star(\rho_{*}(W_{(k)})\wedge\star d\Upsilon_{(l)}) - 2\star(\rho_{*}(W_{(l)})\wedge\star d\Upsilon_{(k)}), 
    \end{align*}
and the residual terms are
    \begin{align*}
R_{(kl)} 
&= 
-\J_{\rho}(\rho_{*}(W_{(k)})\Upsilon_{(l)},\Phi)
-\J_{\rho}(\rho_{*}(W_{(k)})\Phi,\Upsilon_{(l)})
\\&\qquad
-\J_{\rho}(\rho_{*}(W_{(l)})\Upsilon_{(k)},\Phi)
-\J_{\rho}(\rho_{*}(W_{(l)})\Phi,\Upsilon_{(k)})
\\&\qquad
- 2 \langle  \Phi,   \Upsilon_{(l)} \rangle    \Upsilon_{(k)}  
- 2 \langle  \Phi,   \Upsilon_{(k)} \rangle    \Upsilon_{(l)}   
- 2 \langle \Upsilon_{(l)},  \Upsilon_{(k)} \rangle \Phi, 
\\  
\mathcal{R}_{(kl)} 
&= 
- \star(\rho_{*}(W_{(k)})\wedge\star \rho_{*}(W_{(l)})\Phi) 
- \star(\rho_{*}(W_{(l)})\wedge\star \rho_{*}(W_{(k)})\Phi).
  \end{align*}
  
  The third derivatives of \eqref{eq:perturbed YMH1}--\eqref{eq:perturbed YMH2} address $(W_{(123)}, \Upsilon_{(123)})$,
  \begin{align}\label{eqn : linearized 3-1}
  	\Box_{A,\text{Ad}}W_{(123)}+\J_{\rho}(d_{A}\Upsilon_{(123)},\Phi) + Z_{(123)}  &=N_{(123)}, 
\\\label{eqn : linearized 3-2}
  	\Box_{A,\rho}\Upsilon_{(123)}    + \mathcal{Z}_{(123)} &=\mathcal{N}_{(123)},
  \end{align} 
  where $Z_{(123)}$ and $\mathcal{Z}_{(123)}$ are linear and zeroth order in $(W_{(123)},\Upsilon_{(123)})$,
and the nonlinear terms take the form
    \begin{align}\label{NR_123}
N_{(123)} = \tilde N_{(123)} + R_{(123)}, 
\quad
\mathcal N_{(123)} = \tilde{\mathcal N}_{(123)} + \mathcal R_{(123)}, 
    \end{align}
where, denoting by $S_3$ the set of permutations on $\{1, 2, 3\}$,
  \begin{eqnarray*}
  	\tilde{N}_{(123)} &=& -\frac{1}{2} \sum_{\pi \in S_3} \bigg( \frac{1}{2} d^\ast [W_{(\pi(1)\pi(2))}, W_{(\pi(3))}] + \frac{1}{2} d^\ast [W_{(\pi(1))}, W_{(\pi(2)\pi(3))}] \\
  	&&         + \star[W_{(\pi(1)\pi(2))}, \star d W_{(\pi(3))}] + \star[W_{(\pi(1))}, \star d W_{(\pi(2)\pi(3))}]\\
  	&&    + \J_{\rho}(d\Upsilon_{(\pi(1)\pi(2))},\Upsilon_{(\pi(3))})  + \J_{\rho}(d\Upsilon_{(\pi(1))},\Upsilon_{(\pi(2)\pi(3))}) \\
  	&&        + 2\star[W_{(\pi(1))}, \star[W_{(\pi(2))}, W_{(\pi(3))}]] + 2\J_{\rho}(\rho_{*}(W_{(\pi(1))})\Upsilon_{(\pi(2))},\Upsilon_{(\pi(3))})\bigg),  \\
  	\tilde{\mathcal{N}}_{(123)} &=& -  \sum_{\pi \in S_3} \bigg(
  	\star\Big(\rho_{*}(W_{(\pi(1)\pi(2))})\wedge\star d\Upsilon_{(\pi(3))}\Big)   +  \star\Big(\rho_{*}(W_{(\pi(1))})\wedge\star d\Upsilon_{(\pi(2)\pi(3))}\Big)\\
  	&& +  \star\Big(\rho_{*}(W_{(\pi(1))})\wedge\star \rho_{*}(W_{(\pi(2))})\Upsilon_{(\pi(3))}\Big)   +  \langle \Upsilon_{(\pi(1))}, \Upsilon_{(\pi(2))} \rangle \Upsilon_{(\pi(3))} \bigg),
  \end{eqnarray*}
and
  	 	\begin{align*}
R_{(123)} 
&=
-\frac{1}{2} \sum_{\pi \in S_3} \bigg( 
\J_{\rho}(\rho_{*}(W_{(\pi(1)\pi(2))})\Upsilon_{(\pi(3))},\Phi) + \J_{\rho}(\rho_{*}(W_{(\pi(1))})\Upsilon_{(\pi(2)\pi(3))},\Phi) 
\\&\qquad
+ \J_{\rho}(\rho_{*}(W_{(\pi(1)\pi(2))})\Phi,\Upsilon_{(\pi(3))}) + \J_{\rho}(\rho_{*}(W_{(\pi(1))})\Phi,\Upsilon_{(\pi(2)\pi(3))})
)\bigg),
\\
\mathcal{R}_{(123)} 
&=  
-\frac{1}{2} \sum_{\pi \in S_3} \bigg( 
\star\Big(\rho_{*}(W_{(\pi(1)\pi(2))})\wedge\star \rho_{*}(W_{(\pi(3))})\Phi \Big)
\\&\qquad 
+ \star\Big(\rho_{*}(W_{(\pi(1))})\wedge\star \rho_{*}(W_{(\pi(2)\pi(3))})\Phi\Big)
+ 4 \langle \Upsilon_{(\pi(1)\pi(2))},  \Upsilon_{(\pi(3))} \rangle \,\Phi 
\\&\qquad 
+ 4 \langle  \Phi,   \Upsilon_{(\pi(1))} \rangle    \Upsilon_{(\pi(2)\pi(3))}  
+ 4  \langle  \Phi,   \Upsilon_{(\pi(1)\pi(2))} \rangle    \Upsilon_{(\pi(3))} \bigg).
  \end{align*} 
Here $R_{(123)}$ and $\mathcal R_{(123)}$
are quadratic in the variables  (\ref{def_Y})--(\ref{def_Xi})
and contain no derivatives. 
The cubic terms, all of which are of zeroth order in the first place, will be important later and are thus included in 
$\tilde N_{(123)}$ and $\tilde{\mathcal N}_{(123)}$.

\subsection{The threefold linearization and the source-to-solution map}
  
We aim to recover the broken light ray transforms
$\mathbf{S}_{z \gets y \gets x}^{A, \Ad}$
and $\mathbf{S}_{z \gets y \gets x}^{A, \rho}$
on $\mathbb S^+(\mho)$, as defined in Section \ref{sec_S}.
Write 
    \begin{align*}
\cdot^\ast : T_x M \rightarrow T^\ast_x M, \quad
\cdot_\ast : T_x^* M \rightarrow T_x M
    \end{align*}
for the tangent-cotangent isomorphisms induced by the Minkowski metric. Let 
    \begin{align}\label{points_xyz}
(x,y,z) \in \mathbb S^+(\mho).
    \end{align}
After a rotation in the spatial coordinates, we may assume that 
$(x-y)^\ast$, the covector version of $x-y$, is collinear with
    \begin{align}
\label{eqn : xi1} \xi_{(1)} = (1, 1, 0, 0)
    \end{align}
and that $(y-z)^*$ is collinear with 
    \begin{align}
\label{eqn : eta} \eta = (1, -a(r), r, 0),
    \end{align}
where $a(r) = \sqrt{1 - r^2}$ for some $-1 < r < 1$. 
Define for small $s > 0$
    \begin{align}
\label{eqn : xi2} \xi_{(2)} &= (1, a(s), s, 0) \\ \label{eqn : xi3} \xi_{(3)} &= (1, a(s), -s, 0)
    \end{align}
and choose $x_{(k)}$, $k=2,3$, so that $(x_{(k)} - y)^*$ is collinear with $\xi_{(k)}$ and $x_{(k)} \to x$ as $s \to 0$.
Moreover, write $x_{(1)} = x$.

For small $s > 0$ and $\epsilon = (\epsilon_{(1)}, \epsilon_{(2)}, \epsilon_{(3)})$ near $0 \in \mathbb{R}^3$, we apply the source-to-solution map $\mathbf{L}$ to the following sources    
   \begin{align}\label{eqn : source for StS}(J_1, J_2, J_3, \mathcal{F}) &= (J_1(\epsilon, s), J_2(\epsilon, s), J_3(\epsilon, s), \mathcal{F}(\epsilon, s)), \\ \notag J_j(\epsilon, s) &= \epsilon_{(1)} J_{(1), j}(s) + \epsilon_{(2)} J_{(2), j}(s) + \epsilon_{(3)} J_{(3), j}(s),\\ \notag \mathcal{F}(\epsilon, s) &= \epsilon_{(1)} \mathcal{F}_{(1)}(s) + \epsilon_{(2)} \mathcal{F}_{(2)}(s) + \epsilon_{(3)} \mathcal{F}_{(3)}(s), \\ \notag J_{(k), j}(s) &= b_{(k), j} \chi_{(k)} \delta_{x_{(k)}}, \quad k = 1, 2, 3, j = 1, 2, 3,\\ \notag \mathcal{F}_{(k)}(s) &= \upsilon_{(k)} \chi_{(k)} \delta_{x_{(k)}}, \quad k = 1, 2, 3,\end{align} 
   where 
   \begin{itemize}
   	\item $b_{(k), j} \in \mathfrak{g}$ and $\upsilon_{(k)} \in \mathcal{W}$, for $k = 1, 2, 3$ and $j = 1, 2, 3$;
   	\item $\delta_{x_{(k)}}$ is the Dirac delta distribution at $x_{(k)}$, $k = 1, 2, 3$;
   	
   	\item $\chi_{(1)}$ is a microlocal cut-off near $(x_{(1)}, - \xi_{(1)})$, whilst $\chi_{(k)}$ is a microlocal cut-off near $(x_{(k)},   \xi_{(k)})$, $k = 2, 3$; 
   	\item the principal symbol $\sigma[\chi_{(k)}]$ is positively homogeneous of degree $q$ with $q \leq -9$;
   	\item $\mho_{(k)} \cap \mathcal J^+(\mho_{(l)}) = \emptyset$ for all $k \ne l$, where $\mho_{(k)} \subset \mho$ is a neighbourhood  of $x_{(k)}$ containing the supports of $J_{(k),j}$, $j=1,2,3$, $\mathcal F_{(k)}$ and 
   	\begin{align*}
   		\mathcal J^+(\mho_{(k)}) = \{y \in \R^{1+3} :
   		\text{$x < y$ or $x = y$ for some $x \in \mho_{(k)}$} \}; \end{align*}
   	
   	\item $\hat \mho_{(k)} \cap \Gamma_{(l)} = \emptyset$ for all $k \ne l$, where
   	\begin{align*}
   		\hat \mho_{(k)} &= \{(t,x') \in \R^{1+3}: (\tilde t, x') \in \mho_{(k)} \text{ for some $\tilde t \le t$} \},
   		\\
   		\Gamma_{(k)} &= \{x_{(k)} - t\xi_* : t \ge 0,\ (x_{(k)}, \xi) \in \WF(\chi_{(k)}) \}.
   	\end{align*}   	
   \end{itemize}

We shall draw an explicit link between 
the derivative 
    \begin{align*}
\partial_{\epsilon_{(1)}}\partial_{\epsilon_{(2)}}\partial_{\epsilon_{(3)}} \mathbf{L} (J_1, J_2, J_3, \mathcal{F})|_{\epsilon = 0} 
    \end{align*}
of the source-to-solution map $\mathbf{L}$ 
and the threefold linearization \eqref{eqn : linearized 3-1}--\eqref{eqn : linearized 3-2}, at the level of principal symbols at $(z,\eta)$.
Recall that $\mathbf{L}$ was obtained by solving the system \eqref{eq:perturbed YMH1}--\eqref{comp_equiv} 
and then taking the restriction of the perturbed field 
    \begin{align}\label{W2V}
(V, \Psi) = (A + W, \Phi + \Upsilon)
    \end{align}
to $\mho$ in the temporal gauge.
On the other hand, the threefold linearization \eqref{eqn : linearized 3-1}--\eqref{eqn : linearized 3-2} was derived 
from \eqref{eq:perturbed YMH1}--\eqref{eq:perturbed YMH2}
only, that is, the compatibility condition \eqref{comp_equiv}
was neglected in the derivation and $J_0$ was considered as a free variable.
Despite the fact that \eqref{comp_equiv} causes $J_0$ to depend on $\epsilon$ in a nonlinear way, we have 

\begin{lemma} \label{lemma : source reduction}
Suppose that $(W, \Upsilon, J_0)$ solves the system \eqref{eq:perturbed YMH1}--\eqref{comp_equiv} with the source $$(J_1, J_2, J_3, \mathcal{F})$$ as in \eqref{eqn : source for StS}, and denote the derivatives of $(W, \Upsilon)$ in $\epsilon = (\epsilon_{(1)}, \epsilon_{(2)}, \epsilon_{(3)})$ at $\epsilon = 0$ $$(W_{(k)}, \Upsilon_{(k)}, W_{(kl)}, \Upsilon_{(kl)},  W_{(123)}, \Upsilon_{(123)}),$$   as in \eqref{def_Y}--\eqref{def_Xi}. Then
\begin{itemize}
   		\item   $(W_{(k)}, \Upsilon_{(k)})$ solves \eqref{eqn : linearized 1-1}--\eqref{eqn : linearized 1-2} with sources $(J_{(k)}, \mathcal{F}_{(k)})$ 
where
    \begin{align*}
J_{(k)} = \partial_{\epsilon_{(k)}} J_0|_{\epsilon=0} dx^0
+ J_{(k),1} dx^1 + J_{(k),2} dx^2 + J_{(k),3} dx^3;
    \end{align*}
   		\item  
  		 $(W_{(kl)}, \Upsilon_{(kl)},  W_{(123)}, \Upsilon_{(123)})$ solve \eqref{eqn : linearized 2-1}--\eqref{eqn : linearized 3-2} modulo smooth remainder terms.	
\end{itemize}  
   \end{lemma}
   \begin{proof}
The first claim follows simply by differentiating \eqref{eq:perturbed YMH1} with respect to $\epsilon_{(k)}$.
Write
 \begin{eqnarray*}  
 	\vartheta_{(k)} = \frac{\partial J_0}{\partial \epsilon_{(k)}}\bigg|_{\epsilon = 0},
 	&
 	\displaystyle \vartheta_{(kl)} = \frac{\partial^2 J_0}{\partial \epsilon_{(k)}\partial \epsilon_{(l)}}\bigg|_{\epsilon = 0},
 	&
 	\vartheta_{(123)} = \frac{\partial^3 J_0}{\partial \epsilon_{(1)}\partial \epsilon_{(2)}\partial \epsilon_{(3)}}\bigg|_{\epsilon = 0}.    \end{eqnarray*}  
The second claim follows by differentiating \eqref{eq:perturbed YMH1} once we have shown that $\vartheta_{(kl)}$
and $\vartheta_{(123)}$ are smooth.

Recall that \eqref{comp_equiv} is equivalent with \eqref{eq_J0}.
Differentiating \eqref{eq_J0} gives 
    \begin{align}
\label{eqn : first derivative of gauge condition}
\partial_t \vartheta_{(k)} + [A_0, \vartheta_{(k)}]  
&= 
\partial^j J_{(k), j} + [A^j, J_{(k), j}] 
- \mathbb J_\rho(\mathcal F_{(k)}, \Phi),
\\\label{eqn : second derivative of gauge condition}
\partial_t \vartheta_{(kl)} + [A_0, \vartheta_{(kl)}]  
&=  
- [W_{(k), 0}, \vartheta_{(l)}] 
- [W_{(l), 0}, \vartheta_{(k)}]     
+ [W^j_{(l)}, J_{(k), j}] 
+ [W_{(k)}^j, J_{(l), j}]
\\\nonumber&\qquad  
- \mathbb J_\rho(\mathcal F_{(k)}, \Upsilon_{(l)})
- \mathbb J_\rho(\mathcal F_{(l)}, \Upsilon_{(k)}),
    \end{align}
and
    \begin{align}
\label{eqn : third derivative of gauge condition}
&\partial_t \vartheta_{(123)} + [A_0, \vartheta_{(123)}]  
\\\nonumber&\quad=\frac{1}{2} \sum_{\pi \in S_3} \bigg(
- [W_{(\pi(1)\pi(2)), 0}, \vartheta_{(\pi(3))}] 
- [W_{(\pi(1)), 0}, \vartheta_{(\pi(2)\pi(3))}]  
\\\nonumber&\qquad  
+ [W^j_{(\pi(1)\pi(2))}, J_{(\pi(3)), j}] 
- \mathbb J_\rho(\mathcal F_{(\pi(1))}, \Upsilon_{(\pi(2)\pi(3))})
\bigg).
    \end{align}

Let us show that $\vartheta_{(k)}$ is singular only at $x_{(k)}$. 
If we use the notation \[x = (t, x') = (x^0, x^1, x^2, x^3) \in \R^{1+3}, \quad \xi = (\tau, \xi') = (\xi_0, \xi_1, \xi_2, \xi_3) \in T^\ast_x \R^{1+3},\] the differential operator $\partial_t + [A_0, \cdot]$ has the characteristic variety $\{\tau = 0\} \subset T^\ast \R^{1+3}$.
The wavefront set of the right-hand side of (\ref{eqn : first derivative of gauge condition})
is contained in a small conical neighborhood of $(x_{(k)}, \pm \xi_{(k)})$ in $T_{x_{(k)}}^* \R^{1+3}$, and hence it is  disjoint from $\{\tau = 0\}$.
It follows from H\"ormander's propagation of singularities theorem \cite[Th. 26.1.1]{Hormander-Vol4} that $\vartheta_{(k)}$ is singular only at $x_{(k)}$. 

The structure of $\vartheta_{(k)}$ can be described even more precisely. The right-hand side of (\ref{eqn : first derivative of gauge condition}) is a conormal distribution associated to $\{x_{(k)}\}$, with the wavefront set contained in 
	\begin{align}\label{WF_source}
\WF(\chi_{(k)}) \cap T_{x_{(k)}}^* \R^{1+3}.
	\end{align}
As $\partial_t + [A_0, \cdot]$ is elliptic on \eqref{WF_source},
$\vartheta_{(k)}$ is also a conormal distribution associated to $\{x_{(k)}\}$ and $\WF(\vartheta_{(k)})$ is contained in \eqref{WF_source}.

Recall that $(W_{(k)}, \Upsilon_{(k)})$ satisfies the wave equation \eqref{eqn : linearized 1-1}--\eqref{eqn : linearized 1-2}
together with the vanishing initial conditions. 
The source in \eqref{eqn : linearized 1-1}--\eqref{eqn : linearized 1-2} is a conormal distribution associated to $\{x_{(k)}\}$, with the wavefront set contained in \eqref{WF_source}.
It follows that, away from $x_{(k)}$, $(W_{(k)}, \Upsilon_{(k)})$ is a conormal distribution associated to 
    \begin{align}\label{def_K}
K_{(k)} &= \{x_{(k)} - t\xi_* : t > 0,\ \xi \in C_{(k)},\ \xi_\alpha \xi^\alpha = 0 \},
    \end{align}
where $C_{(k)} \subset T_{x_{(k)}}^* \R^{1+3}$ is any conical neighbourhood of the set \eqref{WF_source}.
Moreover, away from $x_{(k)}$, $(W_{(k)}, \Upsilon_{(k)})$
is singular only on $\Gamma_{(k)}$.
The structure of $(W_{(k)}, \Upsilon_{(k)})$ is more complicated near $x_{(k)}$ but its wavefront set is contained in the union of $\Gamma_{(k)}$ and the set \eqref{WF_source}.
We refer to \cite{CLOP}, see Theorem 3, together with Section 2.4 and Appendix B, for a detailed discussion. 

Let us now consider $\vartheta_{(kl)}$.
As $\vartheta_{(k)}$ satisfies \eqref{eqn : first derivative of gauge condition} together with the vanishing initial condition, 
there holds 
    \begin{align*}
\supp(\vartheta_{(k)}) \subset \hat \mho_{(k)} \subset \mathcal J^+(\mho_{(k)}).
    \end{align*}
Hence $\mho_{(k)} \cap \mathcal J^+(\mho_{(l)}) = \emptyset$ implies that the terms 
$[W^j_{(l)}, J_{(k), j}]$ and 
$\mathbb J_\rho(\mathcal F_{(k)}, \Upsilon_{(l)})$
on the right-hand side of \eqref{eqn : second derivative of gauge condition} vanish. The same holds for 
$[W_{(k)}^j, J_{(l), j}]$ and $\mathbb J_\rho(\mathcal F_{(l)}, \Upsilon_{(k)})$
by symmetry, and we have 
    \begin{align}
\label{eqn : second derivative of gauge condition simplified}
\partial_t \vartheta_{(kl)} + [A_0, \vartheta_{(kl)}]  
&=  
- [W_{(k), 0}, \vartheta_{(l)}] 
- [W_{(l), 0}, \vartheta_{(k)}].
    \end{align}
As $W_{(l)}$
is singular only on $\Gamma_{(l)}$, $\hat \mho_{(k)} \cap \Gamma_{(l)} = \emptyset$, we deduce that $W_{(l), 0}$ is smooth in the support of $\vartheta_{(k)}$.
The opposite holds as well, that is, 
$\vartheta_{(k)}$ is smooth in the support of $W_{(l), 0}$.
This follows from $\vartheta_{(k)}$ being singular only at $x_{(k)}$
and $\mho_{(k)} \cap \mathcal J^+(\mho_{(l)}) = \emptyset$. 
In particular, the term $[W_{(l), 0}, \vartheta_{(k)}]$
is smooth. The same holds for $[W_{(k), 0}, \vartheta_{(l)}]$ by symmetry. We conclude that the right-hand side of \eqref{eqn : second derivative of gauge condition simplified} is smooth, and so is then $\vartheta_{(kl)}$.

We have shown that $\vartheta_{(kl)}$ is smooth, 
and turn now to $\vartheta_{(123)}$.
As $\vartheta_{(kl)}$ satisfies \eqref{eqn : second derivative of gauge condition simplified} together with the vanishing initial condition, 
there holds 
    \begin{align*}
\supp(\vartheta_{(kl)}) \subset \hat \mho_{(k)} \cup \hat \mho_{(l)}.
    \end{align*}
Now $\hat \mho_{(k)} \cap \Gamma_{(l)} = \emptyset$,
with $k=\pi(2),\pi(3)$ and $l = \pi(1)$, implies that 
the term 
$[W_{(\pi(1)), 0}, \vartheta_{(\pi(2)\pi(3))}]$ 
on the right-hand side of \eqref{eqn : third derivative of gauge condition}
is smooth.
Moreover, as both $N_{(kl)}$ and $\mathcal N_{(kl)}$
are supported in $\mathcal J^+(\mho_{(k)}) \cap \mathcal J^+(\mho_{(l)})$, it follows from 
\eqref{eqn : linearized 2-1}--\eqref{eqn : linearized 2-2}, with $\vartheta_{(kl)}$ added on the right-hand side of \eqref{eqn : linearized 2-1}, and the vanishing initial conditions 
that $W_{(kl)}$ and $\Upsilon_{(kl)}$
are supported in $\mathcal J^+(\mho_{(k)}) \cup \mathcal J^+(\mho_{(l)})$.
Using $\mho_{(k)} \cap \mathcal J^+(\mho_{(l)}) = \emptyset$
with $k = \pi(3)$ and $l=\pi(1),\pi(2)$
we see that $[W^j_{(\pi(1)\pi(2))}, J_{(\pi(3)), j}]$ vanishes.
Similarly, we have that $\mathbb J_\rho(\mathcal F_{(\pi(1))}, \Upsilon_{(\pi(2)\pi(3))}) = 0$.
Recalling that $\vartheta_{(\pi(3))}$ is singular only at $x_{(3)}$, we see, furthermore, that it is smooth in the support of 
$W_{(\pi(1)\pi(2))}$.
That $[W_{(\pi(1)\pi(2)), 0}, \vartheta_{(\pi(3))}]$
is smooth follows then from 
$\hat \mho_{(k)} \cap \Gamma_{(l)} = \emptyset$,
with $k=\pi(3)$ and $l = \pi(1),\pi(2)$,
and the fact that 
$W_{(kl)}$ may be singular only on 
$\Gamma_{(k)} \cup \Gamma_{(l)}$. 
Before showing this fact, let us point out that, as all the terms on the right-hand side of \eqref{eqn : third derivative of gauge condition} are smooth, so is $\vartheta_{(123)}$.

We finish the proof by showing that $W_{(kl)}$ may be singular only on 
$\Gamma_{(k)} \cup \Gamma_{(l)}$.
Recall that $(W_{(k)}, \Upsilon_{(k)})$ is a conormal distribution associated to $K_{(k)}$ and singular only on $\Gamma_{(k)}$.
It follows from \cite[Th. 8.2.10]{H1} that
    \begin{align*}
\WF(N_{(kl)}) \subset N^* K_{(k)} \cup N^* K_{(l)}
\cup N^* (K_{(k)} \cap K_{(l)})
    \end{align*}
away from $x_{(k)}$ and $x_{(l)}$.
But all lightlike directions in $N^* (K_{(k)} \cap K_{(l)})$ are contained in $N^* K_{(k)} \cup N^* K_{(l)}$, and \cite[Th. 26.1.1]{Hormander-Vol4} 
implies that $W_{(kl)}$ can be singular only on 
$K_{(k)} \cup K_{(l)}$.
As $C_{(k)}$ is an arbitrary conical neighbourhood of $\WF(\chi_{(k)}) \cap T_{x_{(k)}}^* \R^{1+3}$, we see that $W_{(kl)}$ can be singular only on 
$\Gamma_{(k)} \cup \Gamma_{(l)}$.
   \end{proof}

We showed in \cite{CLOP}
that the future flowout of $H_{\sigma[\Box]}$ from
    \begin{align}\label{flowout_generator}
N^* (K_{(1)} \cap K_{(2)} \cap K_{(3)}) \cap \{\sigma[\Box] = 0\}
    \end{align}
is the conormal bundle of a smooth manifold away from $K_{(1)} \cap K_{(2)} \cap K_{(3)}$, and that $(z,\eta)$ belongs to this bundle. Here $K_{(k)}$ is as in \eqref{def_K}, $z$ as in (\ref{points_xyz}),
and $\eta$ as in (\ref{eqn : eta}).
Analogously to \cite{CLOP}, it follows that 
$W_{(123)}$ and $\Upsilon_{(123)}$, as in Lemma \ref{lemma : source reduction}, are conormal distributions associated to this manifold near the point $z$,
and we may consider their principal symbols at $(z,\eta)$.

   \begin{proposition}  \label{prop:from StS to linearized waves}
Let $(J_1, J_2, J_3, \mathcal{F})$ be the source in \eqref{eqn : source for StS} and let $W_{(123)}$ and $\Upsilon_{(123)}$ be as in Lemma \ref{lemma : source reduction}. Then, writing 
\begin{align*}
b_{(0)} &= 
- \frac{\eta_\beta}{\eta_0} \sigma[W_{(123), 0}](z, \eta)  + \sigma[W_{(123),\beta}](z, \eta), 
\\
\upsilon_{(0)} &= \sigma[\Psi_{(123)}](z, \eta), 
   		\end{align*}
there holds
    \begin{align*}
\sigma[\partial_{\epsilon_{(1)}}\partial_{\epsilon_{(2)}}\partial_{\epsilon_{(3)}} \mathbf{L}(J_1, J_2, J_3, \mathcal{F})|_{\epsilon = 0}](z, \eta) = (b_{(0)}, \upsilon_{(0)}).
    \end{align*}
   \end{proposition}    
   \begin{proof}
Recall that $\mathbf{L}(J_1, J_2, J_3, \mathcal{F}) = \mathscr T(V, \Psi)|_\mho$,
where $V$ and $\Psi$ are given by (\ref{W2V}),
and that 
    \begin{align}\label{TVPsi}
\mathscr{T}(V, \Psi) = (\U^{-1}d\U + \U^{-1} V \U, \rho(\U^{-1}) \Psi),
    \end{align}
where $\U$ satisfies
\begin{eqnarray*}\partial_t \U = - V_0 \U, \quad \U|_{t = |x| - 1} = \id.   \end{eqnarray*} 
    By transforming to the temporal gauge, one can always assume that $A|_\mho$ is in the temporal gauge.
   Also due to $V|_{\epsilon=0} = A$, we have that
   $\U|_{\epsilon = 0} = \id$ in $\mho$.
   We write
   \begin{eqnarray*}
   	V_{(k)} = \frac{\partial V}{\partial \epsilon_{(k)}}\bigg|_{\epsilon = 0},
   	&
   	\displaystyle	V_{(kl)} = \frac{\partial^2 V}{\partial \epsilon_{(k)}\partial \epsilon_{(l)}}\bigg|_{\epsilon = 0},
   	&
   	V_{(123)} = \frac{\partial^3 V}{\partial \epsilon_{(1)}\partial \epsilon_{(2)}\partial \epsilon_{(3)}}\bigg|_{\epsilon = 0}, \\ 	\Psi_{(k)} = \frac{\partial \Psi}{\partial \epsilon_{(k)}}\bigg|_{\epsilon = 0},
   	&
   	\displaystyle	\Psi_{(kl)} = \frac{\partial^2 \Psi}{\partial \epsilon_{(k)}\partial \epsilon_{(l)}}\bigg|_{\epsilon = 0},
   	&
   	\Psi_{(123)} = \frac{\partial^3 \Psi}{\partial \epsilon_{(1)}\partial \epsilon_{(2)}\partial \epsilon_{(3)}}\bigg|_{\epsilon = 0},
   \end{eqnarray*} 
and
   \begin{eqnarray*}
   	U_{(k)} = \frac{\partial \U}{\partial \epsilon_{(k)}}\bigg|_{\epsilon = 0},
   	&\displaystyle
   	U_{(kl)} = \frac{\partial^2 \U}{\partial \epsilon_{(k)}\partial \epsilon_{(l)}}\bigg|_{\epsilon = 0},
   &
   	U_{(123)} = \frac{\partial^3 \U}{\partial \epsilon_{(1)}\partial \epsilon_{(2)}\partial \epsilon_{(3)}}\bigg|_{\epsilon = 0},\\ 
   	  	U_{(k)}^{-1} = \frac{\partial \U^{-1}}{\partial \epsilon_{(k)}}\bigg|_{\epsilon = 0},
   	&\displaystyle
   	U_{(kl)}^{-1} = \frac{\partial^2 \U^{-1}}{\partial \epsilon_{(k)}\partial \epsilon_{(l)}}\bigg|_{\epsilon = 0},
   	&
   	U_{(123)}^{-1} = \frac{\partial^3 \U^{-1}}{\partial \epsilon_{(1)}\partial \epsilon_{(2)}\partial \epsilon_{(3)}}\bigg|_{\epsilon = 0}.
   \end{eqnarray*}

Let us show that $U_{(k)}$ is smooth near $z$. As $A_0 = 0$, we have 
    \begin{align*}
\p_t U_{(k)} = -V_{(k),0}.
    \end{align*}
Recall that the wavefront set of $W_{(k)}$ 
is contained in $N^* K_{(k)} \cup C_{(k)}$ 
where $C_{(k)}$ is as in \eqref{def_K}.  
As $A$ is smooth, the same holds for $V_{(k)}$.
The characteristic set $\{\tau = 0\}$ of $\p_t$ 
is disjoint from $N^* K_{(k)} \cup C_{(k)}$, and it follows from 
\cite[Th. 26.1.1]{Hormander-Vol4} that $U_{(k)}$ can be singular only on $K_{(k)}$ or at $x_{(k)}$. In particular, $U_{(k)}$ is smooth near $z$. 

A similar argument shows that $U_{(kl)}$ is smooth near $z$.
Indeed,
    \begin{align*}
\p_t U_{(kl)} = -V_{(kl),0} - V_{(k),0} U_{(l)} - V_{(l),0} U_{(k)}.
    \end{align*}
From the proof of Lemma \ref{lemma : source reduction}
we know that 
    \begin{align*}
\WF(W_{(kl)}) \subset N^* K_{(k)} \cup N^* K_{(l)} \cup N^* (K_{(k)} \cap K_{(l)}) \cup C_{(k)} \cup C_{(l)}.
    \end{align*}
As the point $y$ in \eqref{points_xyz} is outside $\mho$,
we have
    \begin{align*}
\WF(W_{(kl)}|_\mho) \subset N^* K_{(k)} \cup N^* K_{(l)} \cup C_{(k)} \cup C_{(l)},
    \end{align*}
when the microlocal cut-offs in \eqref{eqn : source for StS} are small enough. As above, we see that $V_{(kl),0}$ does not give rise to any singularity in $U_{(kl)}$ near $z$. Arguing as in the proof of Lemma \ref{lemma : source reduction} we also see that 
    \begin{align*}
\WF(V_{(k),0} U_{(l)}|_\mho) \cup \WF(V_{(l),0} U_{(k)}|_\mho)
\subset N^* K_{(k)} \cup N^* K_{(l)} \cup C_{(k)} \cup C_{(l)},
    \end{align*}
and we conclude that $U_{(kl)}$ is smooth near $z$.
As $U_{(k)}$ and $U_{(kl)}$ are smooth near $z$, so are $U_{(k)}^{-1}$ and $U_{(kl)}^{-1}$.

Differentiating \eqref{TVPsi} and writing 
   $$
   (T, \Theta) = \frac{\partial^3 \mathscr T(V,\Psi)}{\partial \epsilon_{(1)}\partial \epsilon_{(2)}\partial \epsilon_{(3)}}\bigg|_{\epsilon = 0}
   $$
we get
   \begin{eqnarray*}
   	T& = &  d U_{(123)} + U^{-1}_{(123)} A +  A U_{(123)} + V_{(123)} +
   	\\ &&\frac 1 2 \sum_{\pi \in S_3} \bigg(   U^{-1}_{(\pi(1) \pi(2))} V_{(\pi(3))}    + U^{-1}_{(\pi(1))} V_{(\pi(2) \pi(3))}    +   V_{(\pi(1) \pi(2))} U_{(\pi(3))} \\& &+    V_{(\pi(1))} U_{(\pi(2) \pi(3))} + U^{-1}_{(\pi(1) \pi(2))} d U_{(\pi(3))} + U^{-1}_{(\pi(1))} d U_{(\pi(2) \pi(3))} \\ && + U^{-1}_{(\pi(1) \pi(2))} A U_{(\pi(3))}  + U^{-1}_{(\pi(1))} A U_{(\pi(2) \pi(3))} \bigg), \\
   	\Theta& = & \rho_\ast(U^{-1}_{(123)}) \Phi +  \Psi_{(123)}  + \frac{1}{2} \sum_{\pi \in S_3} \bigg(	\rho_\ast(U^{-1}_{(\pi(1)\pi(2))}) \Psi_{(\pi(3))} + \rho_\ast(U^{-1}_{(\pi(1))}) \Psi_{(\pi(2)\pi(3))} \bigg),
   \end{eqnarray*}
   where $U_{(123)}$ solves
   \begin{eqnarray*}
   	\p_t U_{(123)} =  - V_{(123),0}  - \frac 1 2 \sum_{\pi \in S_3} \bigg(V_{(\pi(1)\pi(2)),0} U_{(\pi(3))} + V_{(\pi(1)),0} U_{(\pi(2)\pi(3))}\bigg).
   \end{eqnarray*}
   In addition, $\U^{-1} \U = \id$ implies
   $$
   U^{-1}_{(123)} + \frac 1 2 \sum_{\pi \in S_3} \bigg(U^{-1}_{(\pi(1)\pi(2))} U_{(\pi(3))} +  U^{-1}_{(\pi(1))} U_{(\pi(2)\pi(3))} \bigg) + U_{(123)} = 0.
   $$
   Therefore, modulo smooth terms, there holds near $z$
 \begin{equation} 
 \label{dotV_YMH} \left\{ \begin{array}{l}
   	T  =     dU_{(123)}
   	- U_{(123)} A + A U_{(123)} + V_{(123)}, \\  \Theta = \rho_\ast(U^{-1}_{(123)}) \Phi +  \Psi_{(123)},\\
 \p_t U_{(123)} = -V_{(123),0},\\
 U_{(123)}^{-1} = -U_{(123)}.
   \end{array} \right.\end{equation}
Recall that near $z$ it holds that $(V_{(123)}, \Psi_{(123)})$ is a conormal distribution associated to the future flowout of (\ref{flowout_generator}).
   As the flowout is contained in the characteristic set of $\Box$, it is disjoint from the characteristic set $\{\tau = 0\}$ of $\p_t$. The third equation in \eqref{dotV_YMH} implies that $U_{(123)}$ is a conormal distribution associated to the same flowout near $z$.
   Then taking principal symbols in \eqref{dotV_YMH} gives for $\beta = 0,1,2,3$,
 \begin{equation}\label{dotV_YMH_symbol}
\left\{  
\begin{array}{l}
   	\sigma[T_\beta](z, \eta) = \imath \eta_\beta \sigma[U_{(123)}] (z, \eta) + \sigma[W_{(123),\beta}](z, \eta), \\  \sigma[\Theta](z, \eta) = -\sigma[\rho_\ast(U_{(123)})](z, \eta) \Phi +   \sigma[\Upsilon_{(123)}](z, \eta),
   	\\
   	\imath \eta_0  \sigma[U_{(123)}](z, \eta) = - \sigma[W_{(123),0}](z, \eta).
   \end{array} \right. \end{equation}  
  Finally, we conclude the proof by solving for $\sigma[U_{(123)}](z, \eta)$ in the last equation of \eqref{dotV_YMH_symbol}  and substituting in the first two equations of \eqref{dotV_YMH_symbol}.
   \end{proof}
    
\section{Recovery of the broken light ray transform in the adjoint representation}   \label{section:recoveryad}
As pointed out in Subsection \ref{subsection:outline}, we can recover $\mathbf{S}_{z \gets y \gets x}^{A, \Ad}$ by activating only the YM-channel. This strategy will decouple the field $W$ from the field $\Upsilon$ on the principal level such that one can employ the method established in \cite{CLOP2}.
The idea is to choose the source in \eqref{eqn : source for StS} so that $\upsilon_{(k)} = 0$ and $b_{(k)}$ is as in \cite{CLOP2}.
The equation \eqref{comp_equiv} differs by the terms 
$\mathbb J_\rho(\mathcal F, \Upsilon)$ and $\mathbb J_\rho(\mathcal F, \Phi)$ from the analogous equation in \cite{CLOP2}, cf. Lemma 4 there, however, both of them lead to the same equation for the principal symbols
    \begin{align*}
\sigma[\p_{\epsilon_{(k)}} J_0|_{\epsilon = 0}](x_{(k)}, \pm \xi_{(k)}),
\quad k=1,2,3.
    \end{align*}
More precisely, differentiating and computing the principal symbols in \eqref{comp_equiv}, we get
    \begin{align*}
\xi_{(k),0} \sigma[\p_{\epsilon_{(k)}} J_0|_{\epsilon = 0}](x_{(k)}, \pm \xi_{(k)}) = \xi_{(k)}^j b_{(k),j}.
    \end{align*}
Here we used the specific form of the source in \eqref{eqn : source for StS}.
Then, following \cite{CLOP2}, we choose $b_{(k),j} = 0$ for $j=1,3$
and let $b_{(k),2} = b_{(k)} \in \mathfrak g$.
Recalling that $\xi_{(1)}$, $\xi_{(2)}$ and $\xi_{(3)}$ are defined by \eqref{eqn : xi1}, \eqref{eqn : xi2} and \eqref{eqn : xi3},
we see that $J_{(k)}$ in Lemma \ref{lemma : source reduction} satisfies
    \begin{align*}
\sigma[J_{(k)}](x_{(k)}, \pm \xi_{(k)})
= b_{(k)} \omega_{(k)},
    \end{align*}
where
    \begin{align}\label{def_w}
\omega_{(1)} = dx^2,
\quad
\omega_{(k)} = (-1)^k s dx^0 + dx^2,
\quad
k=2,3.
    \end{align}

As the parallel transport equations (\ref{eq:imp})--(\ref{eq:imp2}) for the linearized system 
(\ref{eqn : linearized 1-1})--(\ref{eqn : linearized 1-2})
are of the upper triangular form and as $\upsilon_{(k)} = 0$, we see that 
    \begin{align}\label{symbol_at_y}
\sigma[(W_{(k)}, \Upsilon_{(k)})](y,\pm\xi_{(k)}) = (\alpha_{(k)} \P_{y \gets x_{(k)}}^{A,\Ad} b_{(k)} \omega_{(k)}, 0),
    \end{align}
where $\alpha_{(k)}$ is the same volume factor as in \cite[Section 8.2]{CLOP2}.
As $\Upsilon_{(k)}$ vanishes on the principal level, the nonlinear interaction terms 
    \begin{align*}
N_{(kl)}, \quad
\mathcal N_{(kl)}, \quad
N_{(123)}, \quad
\mathcal N_{(123)}
    \end{align*}
reduce to the same terms as in \cite{CLOP2}.
Taking $b_{(3)} = b_{(2)}$, the proof of  \cite[Proposition 8]{CLOP2} shows that, with the above choice of $J_2$,
    \begin{align*}
\sigma[\partial_{\epsilon_{(1)}}\partial_{\epsilon_{(2)}}\partial_{\epsilon_{(3)}} \mathbf{L}(0, J_2, 0, 0)|_{\epsilon = 0}](z, \eta)
    \end{align*}
determines $\mathbf{S}_{z \gets y \gets x}^{A, \Ad}[b_{(2)}, [b_{(1)}, b_{(2)}]]$. We claim that this is enough to determine 
$\mathbf{S}_{z \gets y \gets x}^{A, \Ad}$. To see this consider the following orthogonal decomposition with respect to the $\text{Ad}$-invariant inner product:
	\begin{align}\label{splitting}
\g=Z(\g)\oplus Z(\g)^{\perp}.
	\end{align}
For brevity we set $\g_{1}=Z(\g)^{\perp}$ and note that $\g_{1}$ is a Lie algebra with trivial centre. From the definitions it is straightforward to check that $\mathbf{S}_{z \gets y \gets x}^{A, \Ad}$ leaves $\g_{1}$ invariant and it is the identity on the centre $Z(\g)$. By \cite[Proposition 9]{CLOP2}, $\mathbf{S}_{z \gets y \gets x}^{A, \Ad}[b_{(2)}, [b_{(1)}, b_{(2)}]]$ determines $\mathbf{S}_{z \gets y \gets x}^{A, \Ad}$ on $\g_1$ and and hence we recover the whole $\mathbf{S}_{z \gets y \gets x}^{A, \Ad}$ as claimed.
We emphasize that even if $\Ad$ is faithful, we will be able to recover $A$ in $\mho$ only up to a gauge since 
we have actually replaced $A$ with its gauge equivalent copy, see the remarks after Lemma \ref{lem_pass_temp_gauge} and compare with Corollary \ref{corollary:rhofaithful}.

 \section{Perturbing coupled fields with abelian sources} \label{section:abelian}

We will now show how the abelian components in the Yang--Mills--Higgs system can be employed to recover $\mathbf{S}_{z \gets y \gets x}^{A, \rho}$ and eventually also $\Phi$.
The computations are first illustrated in the case that $G = U(1)$,
the general case being similar to this.

\subsection{The abelian case}\label{subsec : Maxwell-Higgs} 

As a warm-up let us consider the case that the gauge group $G$ is simply $U(1)$ and that $\dim \mathcal W = 1$. In this case, 
any linear representation is of the form 
    \begin{align}\label{def_rho_abelian}
\rho(e^{\imath \theta}) = e^{\imath n\theta}
    \end{align}
for some $n \in \mathbb Z$. Since we are assuming that $\rho_{*}$ has trivial kernel we must have $n\neq 0$.
We take $n = 1$ for simplicity, that is, $\rho(e^{\imath \theta})$ is multiplication by $e^{\imath \theta}$ on $\mathcal W$.

The gauge fields $A$ are viewed as the electromagnetic fields, and we emphasize their special features by writing $a$ in place of $A$. Note that $\Box_{a,\Ad} = \Box$, and that
the leading interaction terms reduce to 
\begin{align*}
	\tilde{N}_{(kl)} =& - \J_{\rho}(d\Upsilon_{(k)},\Upsilon_{(l)})  - \J_{\rho}(d\Upsilon_{(l)},\Upsilon_{(k)}) \\
	\tilde{\mathcal{N}}_{(kl)} =&  - 2\star(W_{(k)}\wedge\star d\Upsilon_{(l)}) - 2\star (W_{(l)}\wedge\star d\Upsilon_{(k)}), \\
		\tilde{N}_{(123)} =& -\frac{1}{2} \sum_{\pi \in S_3} \bigg( \J_{\rho}(d\Upsilon_{(\pi(1)\pi(2))},\Upsilon_{(\pi(3))})  + \J_{\rho}(d\Upsilon_{(\pi(1))},\Upsilon_{(\pi(2)\pi(3))}) \\
	&    + 2\J_{\rho}(W_{(\pi(1))}\Upsilon_{(\pi(2))},\Upsilon_{(\pi(3))})\bigg) \\
	\tilde{\mathcal{N}}_{(123)} =& -  \sum_{\pi \in S_3} \bigg(
	\star\Big(W_{(\pi(1)\pi(2))}\wedge\star d\Upsilon_{(\pi(3))}\Big)   +  \star\Big(W_{(\pi(1))}\wedge\star d\Upsilon_{(\pi(2)\pi(3))}\Big)\\
	& +  \star\Big(W_{(\pi(1))}\wedge\star W_{(\pi(2))}\Upsilon_{(\pi(3))}\Big)   +   \langle \Upsilon_{(\pi(1))}, \Upsilon_{(\pi(2))} \rangle   \Upsilon_{(\pi(3))} \bigg) .
\end{align*} 

As in the previous section, we want to simplify the computations by choosing some of the source terms in (\ref{eqn : source for StS}) to vanish. Here a convenient choice is to let $b_{(1),j} = 0$, $j=1,2,3$,
and $\upsilon_{(2)} = \upsilon_{(3)} = 0$.
We let $\upsilon_{(1)} \in \mathcal W$, and, similarly to above,  choose 
$b_{(k),j} = 0$ for $j=1,3$, $k=2,3$,
and let $b_{(k),2} = b_{(k)} \in \mathfrak g = \imath \R \setminus 0$ for $k=2,3$.
Then, for $k=2,3$, $J_{(k)}$ in Lemma \ref{lemma : source reduction} satisfies
    \begin{align*}
\sigma[J_{(k)}](x_{(k)}, \xi_{(k)})
= b_{(k)} \omega_{(k)},
    \end{align*}
where, see \eqref{def_w},
    \begin{align*}
\omega_{(k)} = (-1)^k s dx^0 + dx^2,
\quad
k=2,3.
    \end{align*}
To simplify the notation, we write $\omega_{(1)} = 0$ and $b_{(1)} = 0$.

Recall that near the point $y$, $(W_{(k)}, \Upsilon_{(k)})$
is a conormal distribution associated to the manifold $K_{(k)}$
defined by \eqref{def_K}. We will use the product calculus of conormal distributions due to Greenleaf-Uhlmann \cite[Lemma 1.1]{Greenleaf-Uhlmann-CMP-1993} (see also \cite[Lemma 2]{CLOP}) to compute the principal symbols of the leading interaction terms.
For this reason, we need to compute the coefficients $\kappa_{(k)}$ in the splitting
\begin{align*} & \eta = \eta_{(1)} + \eta_{(2)} + \eta_{(3)} \in N_y^\ast K_{(1)} \oplus N_y^\ast K_{(2)} \oplus N_y^\ast K_{(3)},\\ & \eta_{(k)} = \kappa_{(k)} \xi_{(k)} \quad \mbox{for $k = 1, 2, 3$}.\end{align*} 
They are    
   \begin{align*} 
   	\kappa_{(1)} = 1 - \frac{1 + a(r)}{1 - a(s)},
   	\quad
   	\kappa_{(2)} = \frac{1 + a(r)}{2(1 - a(s))} + \frac{1}{2} \frac{r}{s},
   	\quad
   	\kappa_{(3)} = \frac{1 + a(r)}{2(1 - a(s))} - \frac{1}{2} \frac{r}{s}.
   \end{align*}

Analogously to (\ref{symbol_at_y}), using the homogeneity in Proposition \ref{prop_homogeneity}, we have
    \begin{align*}
\sigma[(W_{(k)}, \Upsilon_{(k)})](y, \eta_{(k)}) 
= 
\alpha_{(k)} |\kappa_{(k)}|^{q-1} \mathbf{P}_{y\gets x_{(k)}}^{a,\Phi,\rho} (b_{(k)} \omega_{(k)},\upsilon_{(k)}),
    \end{align*}
where $q$ is the degree of homogeneity of $\sigma[\chi_{(k)}]$ in \eqref{eqn : source for StS}, and the volume factors $\alpha_{(k)}$ satisfy $\alpha_{(k)} \rightarrow \alpha_{(1)}$ for $k = 2, 3$ as $s \rightarrow 0$. See \cite[(65)]{CLOP} for the precise form of $\alpha_{(k)}$. Let us remark that the volume factors are not determined completely by Proposition \ref{prop_homogeneity} since the analysis of the initial condition for the transport equation \eqref{Lie_transport} is not included there. We refer to  \cite[Theorem 3 and appendices]{CLOP} for a detailed discussion of the initial condition. The analysis in \cite{CLOP} is based on using techniques from \cite{Melrose-Uhlmann-CPAM1979}. 

For $k=2,3$, $\upsilon_{(k)} = 0$ and this reduces to 
    \begin{align}\label{propagation_23}
\sigma[(W_{(k)}, \Upsilon_{(k)})](y, \eta_{(k)}) 
= 
\alpha_{(k)} |\kappa_{(k)}|^{q-1} (\omega_{(k)} b_{(k)}, 0),
    \end{align}
where we used the fact that $\mathbf{P}_{y\gets x_{(k)}}^{a,\Ad} b_{(k)} = b_{(k)}$ in the abelian case. 

To absorb all scalar factors, we introduce the notations,
\begin{align}\label{eqn:W_rescaled_symbol}
	\hat W_{(j)}  &=
	(\alpha_{(j)})^{-1} |\kappa_{(j)}|^{1-q} \sigma[W_{(j)}](y, \eta_{(j)}),
	\\	\hat \Upsilon_{(j)}  &=
	(\alpha_{(j)})^{-1} |\kappa_{(j)}|^{1-q} \sigma[\Upsilon_{(j)}](y, \eta_{(j)}). 
\end{align}
Moreover, we denote rescaled multiple fold linearized quantities,
\begin{align*} \begin{cases}
	\hat \bullet_{(kl)} =
  (\alpha_{(kl)})^{-1} |\kappa_{(k)}\kappa_{(l)}|^{1-q} \sigma[\bullet_{(kl)}](y, \eta_{(kl)}),
	\\ 
	\hat \bullet_{(123)} =
	 \alpha^{-1} |\kappa_{(1)}\kappa_{(2)}\kappa_{(3)}|^{1-q} \sigma[\bullet_{(123)}](y, \eta), 
\end{cases}\quad  \mbox{for any $\bullet \in \{W, \Upsilon, \tilde{N}, \tilde{\mathcal{N}}\}$,}\end{align*}
where  $\eta_{(kl)} = \eta_{(k)} + \eta_{(l)}$,
$\alpha_{(kl)} = \alpha_{(k)}\alpha_{(l)}$, and $\alpha = \iota \alpha_{(1)}\alpha_{(2)}\alpha_{(3)}$ with some $\iota \in \mathbb C \setminus 0$.
\HOX{G: I think before we used $\iota$ for $\sqrt{-1}$}
 
As $\mathbf{P}_{y\gets x_{(k)}}^{A,\Phi,\rho}$ is upper triangular, and as $\upsilon_{(2)} = \upsilon_{(3)} = 0$, there holds 
$\hat \Upsilon_{(2)} = \hat \Upsilon_{(3)} = 0$.
This again implies that $\hat{\tilde N}_{(kl)} = 0$ for all distinct $k,l=1,2,3$. 

\HOX{C: I gave up swapping $N$ and $\tilde{N}$. Not only is it dangerous but it also would be inconsistent with the YM paper, the argument in which we referred the reader to numerous times.}
The leading twofold interaction terms in the Higgs channel read 
$$\hat{\tilde{\mathcal{N}}}_{(kl)} = \left\{ \begin{array}{ll}
 	2 \imath \eta_{(1), \alpha}\hat{W}_{(2)}^\alpha \hat{\Upsilon}_{(1)} = - 2 \imath\kappa_{(1)} s b_{(2)} \hat{\Upsilon}_{(1)}  & k, l = 1,2,
 	\\2 \imath\eta_{(1), \alpha}\hat{W}_{(3)}^\alpha \hat{\Upsilon}_{(1)} = + 2 \imath\kappa_{(1)} s b_{(3)} \hat{\Upsilon}_{(1)} & k, l = 1, 3,
 	\\	0 & \mbox{otherwise}.
 \end{array}  \right.$$
Here we used 
    \begin{align}\label{eta_omega_12_13}
\eta_{(1), \alpha}\omega_{(2)}^\alpha 
= \kappa_{(1)} \xi_{(1), \alpha}\omega_{(2)}^\alpha 
= - \kappa_{(1)} s,
\quad
\eta_{(1), \alpha}\omega_{(3)}^\alpha 
= \kappa_{(1)} \xi_{(1), \alpha}\omega_{(3)}^\alpha 
= \kappa_{(1)} s.
    \end{align}
 
As $\hat{\tilde{\mathcal{N}}}_{(12)} \ne 0$ 
the remainder term $\mathcal R_{(12)}$ 
in (\ref{NR_kl}) does not affect the principal symbol
of $\mathcal{N}_{(12)}$ in a conical neighbourhood of $(y, \eta_{(12)})$.
Considering the principal symbols of $W_{(kl)}$ and $\Upsilon_{(kl)}$
on the same level, in other words, 
viewing them as components of 
    \begin{align*}
\sigma[(W_{(kl)}, \Upsilon_{(kl)})_{k,l=1,2,3}],
    \end{align*}
it follows that $\hat{W}_{(kl)} = 0$ and 
\HOX{L: what happens to $\imath$?
	
C : I no longer follow the convention in the YM paper but removed $1/\imath$ or $1/\imath^2$ in $\hat{\bullet}$.}
\begin{align*}  \hat{\Upsilon}_{(kl)} &= \left\{ \begin{array}{ll}  -2 \imath \kappa_{(1)} s \sigma[\Box]^{-1} (y, \eta_{(21)})  b_{(2)} \hat{\Upsilon}_{(1)}  & k, l = 1,2,
 		\\  + 2 \imath \kappa_{(1)} s \sigma[\Box]^{-1} (y, \eta_{(31)}) b_{(3)} \hat{\Upsilon}_{(1)} & k, l = 1, 3,
 		\\	0 & \mbox{otherwise}.
 	\end{array}  \right.
\end{align*}
Note that since $\eta_{(k)}$ and $\eta_{(l)}$ are both light-like, $\eta_{(kl)}$ cannot be light-like, i.e. 
    \begin{align*}
\sigma[\Box](y, \eta_{(kl)}) \neq 0.
    \end{align*}

Let us show next that $\hat{\tilde{N}}_{(123)} = 0$.
Note that $\hat \Upsilon_{(\pi(1)\pi(2))} \ne 0$
implies that $\pi(1) = 1$ or $\pi(3) = 1$.
But then $\hat \Upsilon_{(\pi(3))} = 0$. 
Thus the first term in $\hat{\tilde{N}}_{(123)}$ is zero.
The second term is zero for the same reason. 
Finally, the third term vanishes since $\hat \Upsilon_{(\pi(2))} = 0$
or $\hat \Upsilon_{(\pi(3))} = 0$.

We turn to $\hat{\tilde{\mathcal N}}_{(123)}$.
As $\hat W_{(kl)} = 0$ for all distinct $k,l = 1,2,3$,
the first term in the sum in $\hat{\tilde{\mathcal N}}_{(123)}$
vanishes. Also the last term vanishes since $\hat \Upsilon_{(2)} = \hat \Upsilon_{(3)} = 0$. We are left with
    \begin{align*}
\hat{\tilde{\mathcal N}}_{(123)}
&=
\hat{W}_{(2)}^\alpha \imath \eta_{(13), \alpha} \hat{\Upsilon}_{(13)} 
+ \hat{W}_{(3)}^\alpha \imath \eta_{(12), \alpha} \hat{\Upsilon}_{(12)} 
+2 \hat{W}_{(2)}^\alpha \hat{W}_{(3), \alpha}\hat{\Upsilon}_{(1)}
\\&=
( 
-2 \imath \kappa_{(1)} (\kappa_{(1)} + 2 \kappa_{(3)}) s^2 \sigma[\Box]^{-1} (y, \eta_{(13)})
- 2 \imath \kappa_{(1)} (\kappa_{(1)} + 2 \kappa_{(2)}) s^2
\sigma[\Box]^{-1} (y, \eta_{(12)})
\\&\qquad + 2 (1 + s^2)
) b_{(2)} b_{(3)} \hat \Upsilon_{(1)},
    \end{align*}
where we used (\ref{eta_omega_12_13}) and 
    \begin{align*}
\eta_{(2), \alpha}\omega_{(3)}^\alpha 
= 2\kappa_{(2)} s,
\quad
\eta_{(3), \alpha}\omega_{(2)}^\alpha 
= -2\kappa_{(3)} s,
\quad
\omega_{(2), \alpha}\omega_{(3)}^\alpha 
= 1 + s^2.
    \end{align*}
Using also 
    \begin{align*}
\sigma[\Box]^{-1}(y, \eta_{(1 k)}) = 
\frac{1}{2(a(r) + a(s)) \kappa_{(k)}}, \quad
k = 2,3,
    \end{align*}
we can verify that
    \begin{align*}
\hat{\tilde{\mathcal N}}_{(123)}
= 2 b_{(3)} b_{(2)} \hat{\Upsilon}_{(1)} + \mathcal O(s).
    \end{align*}
This quantity is non-vanishing for small enough $s$,
and it is straightforward to verify that the remainder terms
in (\ref{NR_123})
do not affect the principal symbol
$$
\sigma[(N_{(123)}, \mathcal N_{(123)})]
$$
in a conical neighbourhood of $(y, \eta)$.
To summarize, taking the limit $s \rightarrow 0$, there holds 
\begin{align}\label{threefold_symbol_at_y}
\lim_{s \to 0} (\hat{W}_{(123)}, \hat{\Upsilon}_{(123)})  =  \left(0,  2 b_{(3)} b_{(2)}\hat{\Upsilon}_{(1)}     \right).  
\end{align}

Recall that $(W_{(123)},\Upsilon_{(123)})$ are conormal distributions associated to the future flowout from \eqref{flowout_generator}. As the linear parts of \eqref{eqn : linearized 3-1}--\eqref{eqn : linearized 3-2} and \eqref{eqn : linearized 1-1}--\eqref{eqn : linearized 1-2}
coincide, modulo zeroth order terms, 
the principal symbol of $(W_{(123)}, \Upsilon_{(123)})$ is also propagated by 
$\P_{z \gets y}^{a, \Phi, \rho}$.
We omit again the detailed description of the volume factor $\alpha_{(0)} \in \C \setminus \{0\}$ below, and refer to \cite{CLOP} for details. We have
\HOX{Is there really 2 here? Where does - come from? A power of $\imath$ perhaps?

C: Fixed. Since only the $(1 + 1 + 1)$-fold terms matter, there should be no $\imath$.}
 \begin{align} 
\label{eqn : broken lightray EM}  \lim_{s \to 0} C_\alpha \sigma[W_{(1 2 3)}](z, \eta)  &=	  C_b \left(\P_{z\gets y}^{a, \Phi, \rho}\right)_{12} \P_{y\gets x }^{a, \rho} \upsilon_{(1)},
 	\\
\label{eqn : broken lightray H abelian}  \lim_{s \to 0} C_\alpha \sigma[\Upsilon_{(1 2 3)}](z, \eta)  &=    C_b     \P_{z \gets y}^{a, \rho} \P_{y \gets x}^{a, \rho} \upsilon_{(1)},  \end{align}  
where we employ the following shorthands  
\begin{align*}&C_\alpha = \alpha_{(0)}^{-1}   \alpha^{-1} |\kappa_{(1)}\kappa_{(2)}\kappa_{(3)}|^{1-q}, \quad  C_b  = 2 b_{(3)} b_{(2)}.\end{align*}  
Here $C_b \ne 0$ as $b_{(k)} \ne 0$ for $k=2,3$.
We can read from  \eqref{eqn : broken lightray H abelian} the broken light ray transform $\mathbf{S}^{a, \rho}_{z \gets y \gets x} \upsilon_{(1)} = \P_{z \gets y}^{a, \rho} \P_{y \gets x}^{a, \rho} \upsilon_{(1)}$.

The assumption $Z(\mathfrak{g}) \cap \ker \rho_\ast = \{0\}$ implies $\rho_\ast$ is faithful in the abelian case, hence Corollary \ref{corollary:rhofaithful} recovers $a$ from $\mathbf{S}^{a, \rho}_{z \gets y \gets x}$ up to a gauge transformation $\U$ satisfying $\U|_\mho = \id$. (But the gauge is not fully fixed in $\mho$ as we have already replaced $a$ with its gauge equivalent copy, see the remarks after Lemma \ref{lem_pass_temp_gauge}.)

As 
$\left(\P_{z\gets y}^{a, \Phi, \rho}\right)_{12} \P_{y\gets x }^{a, \rho} = \left(\P_{z\gets y}^{(a, \Phi) \cdot \U, \rho}\right)_{12} \P_{y\gets x }^{a \cdot \U, \rho}$
whenever $\U$ satisfies $\U|_\mho = \id$, 
we may fix a gauge equivalent copy $a \cdot \U$ of $a$, and replace $(a, \Phi)$ by $(a, \Phi) \cdot \U$.
In other words, as we know the gauge orbit of $a$ and $\left(\P_{z\gets y}^{a, \Phi, \rho}\right)_{12} \P_{y\gets x }^{a, \rho}$ depends only on the orbit, we simply choose one point on the orbit.
We will proceed to compute the corresponding $\Phi$.

As we can choose $\P_{y \gets x_{(1)}}^{a, \rho} \upsilon_{(1)}$  arbitrarily in virtue of the surjectivity of $\P_\gamma^{a, \rho}$, it suffices to reconstruct $\Phi$ from the knowledge of $\left(\P_{z\gets y}^{a, \Phi, \rho}\right)_{12}$.  
Recall the definition of  $\left(\P_{z\gets y}^{a, \Phi, \rho}\right)_{12}$, \[ \left(\P_{z\gets y}^{a, \Phi, \rho}\right)_{12}(v) =  \frac12 \P_{z \gets y}^{a, \Ad} \int_{t_y}^{t_z} \J_\rho(\dot{\gamma}_\beta(s) v, \rho(\U_\gamma^a(s))^{-1}\Phi(\gamma(s))) \, ds,\] where $\gamma$ is the parametrization of the line satisfying $\gamma(t_y) = y$ and $\gamma(t_z) = z$.
As $\P_{z \gets y}^{a, \Ad}$ is the identity map, and as $\dot \gamma$ is a constant, one can determine the quantity $$\dot{\gamma}_\beta \J_\rho\bigg(v, \int_{t_y}^{t_z}\rho(\U_\gamma^a(s))^{-1}\Phi(\gamma(s))\, ds\bigg).$$ Then the non-degeneracy of $\J_\rho$ further gives the integral \[\mathcal{I} = \int_{t_y}^{t_z}\rho(\U_\gamma^a(s))^{-1}\Phi(\gamma(s))\, ds.\] 
Finally, we recover
    \begin{align*}
\Phi(\gamma(t_y)) = - \rho(\U_\gamma^a(t_y)) \partial_{t_y} \mathcal{I}.
    \end{align*}

We emphasize that 
by Proposition \ref{prop:from StS to linearized waves} the symbols
\eqref{eqn : broken lightray EM} and \eqref{eqn : broken lightray H abelian}
are given by the threefold linearization of the source-to-solution map $\mathbf{L}$, which is determined by the data set $\mathcal{D}_{(a, \Phi)}$ as is proven in Proposition \ref{prop : source-to-solution map}.
In summary, we have proved the following special case of Theorem \ref{thm:main thm}
\begin{proposition}\label{prop:emh}
Suppose that $G = U(1)$, $\dim \mathcal W = 1$ and that \eqref{def_rho_abelian} holds with $n=1$.
If $(a, \Phi)$ and $(b, \Xi)$ 
satisfy \eqref{eq:ymh1} and \eqref{eq:ymh2}
in $\mathbb{D}$, then  $$\mathcal{D}_{(a, \Phi)} = \mathcal{D}_{(b, \Xi)} \Longleftrightarrow (a, \Phi) \sim (b, \Xi) \, \mbox{in $\mathbb{D}$}.$$
\end{proposition}

  \subsection{Completion of the recovery}\label{subsec : proof of main theorem}

We will now consider the case of general $G$ and $\rho$
satisfying the assumptions of Theorem \ref{thm:main thm}.
Let us follow the computation in the previous section as closely as possible, and choose again sources so that $b_{(1),j} = 0$, $j=1,2,3$,
and $\upsilon_{(2)} = \upsilon_{(3)} = 0$.
We let $\upsilon_{(1)} \in \mathcal W$, and 
$b_{(k),j} = 0$ for $j=1,3$, $k=2,3$.
Finally, the key idea is to take $b_{(k),2}$, $k=2,3$, in the center $Z(\mathfrak g)$, that is, we let $b_{(k),2} = b_{(k)} \in Z(\mathfrak g)$ be nonzero for $k=2,3$. Here we are using the assumption that $Z(\mathfrak g)$ is non-trivial.

Now $\upsilon_{(k)} = 0$ and $b_{(k)} \in Z(\mathfrak g)$ for $k=2,3$
imply again \eqref{propagation_23}. 
In particular, using the notation from the previous section,
$\hat W_{(k)} \in Z(\mathfrak g)$ for $k=2,3$.
This again implies that the commutator terms in 
$\hat{\tilde N}_{(kl)}$ vanish since at least one of the factors in each commutator is in the abelian component. 
The other terms vanish as well by the argument in the previous section, and we have $\hat{\tilde N}_{(kl)} = 0$ for all distinct $k,l = 1,2,3$.
Furthermore, $\hat{\tilde{\mathcal N}}_{(kl)}$ does not contain any commutator terms in the first place, and it is treated as in the previous section. 

We turn to the there-fold interaction terms. 
The first four commutator terms in $\hat{\tilde N}_{(123)}$
contain two-fold interaction terms, and they vanish since $\hat{\tilde N}_{(kl)} = 0$. Also, the final commutator term vanishes since $\hat W_{(k)} \in Z(\mathfrak g)$ for $k=2,3$.
Thus we are reduced to the same interactions as in the previous section, and see that \eqref{threefold_symbol_at_y} holds.
Hence 
 \begin{align*} 
 \lim_{s \to 0} C_\alpha \sigma[W_{(1 2 3)}](z, \eta)  &=	   C_b \left(\P_{z\gets y}^{A, \Phi, \rho}\right)_{12} \P_{y\gets x}^{A, \rho} \upsilon_{(1)},
\\
\notag  \lim_{s \to 0} C_\alpha \sigma[\Upsilon_{(1 2 3)}](z, \eta)  &=      C_b     \P_{z \gets y}^{A, \rho} \P_{y \gets x}^{A, \rho} \upsilon_{(1)},  \end{align*} where the shorthand $C_b$ and $C_\alpha$ are defined as in \eqref{eqn : broken lightray EM}--\eqref{eqn : broken lightray H abelian}. 
Here we used the fact that $b_{(2)}$ and $b_{(3)}$, contained in $C_b$, are in $Z(\g)$. 
It follows from fully chargedness that $C_b$ is non-singular, and we recover $\mathbf S_{z \gets y \gets x}^{A, \rho}$ as before.

Recall that we have already determined $\mathbf S_{z \gets y \gets x}^{A, \Ad}$, and hence we know $\mathbf S_{z \gets y \gets x}^{A, \Ad \oplus \rho}$. The assumption $Z(\g) \cap \mathrm{Ker} \rho_* = \{0\}$
implies that $\Ad \oplus \rho$ is faithful,
and $A$ is determined, up to a gauge transformation, due to Corollary \ref{corollary:rhofaithful}.
Finally, $\Phi$ is recovered analogously to the previous section using the fact that $\rho_{*}$ is fully charged (which is equivalent to $\J_{\rho}$ being non-degenerate).
This completes the proof of Theorem \ref{thm:main thm}.
	
	\bigskip
	
	\noindent {\bf Acknowledgements.} XC was supported by Natural Science Foundation of Shanghai grant 23JC1400501.   ML was supported by Reseach Council of Finland grants
	320113 and 312119.
	LO was supported by EPSRC grants EP/P01593X/1 and EP/R002207/1, by the ERC grant 101086697 (LoCal), and the Reseach Council of Finland, grants 347715 and 353096. 
	XC and GPP were supported by EPSRC grant EP/R001898/1. 
Views and opinions expressed are those of the authors only and do not necessarily reflect those of the European Union or the other funding organizations.

	\bigskip	\noindent {\bf Data Availability Statement.} Data sharing not applicable to this article as no datasets were generated or analysed during the current study.
	
	\bigskip	\noindent {\bf Conflict of Interest.} The authors have no conflicts of interest to declare that are relevant to the content of this article.

\bibliographystyle{abbrv}
\bibliography{main}

\begin{thebibliography}{10}

\bibitem{Alinhac1983}
S.~Alinhac.
\newblock Non-unicit\'e du probl\`eme de {C}auchy.
\newblock {\em Ann. of Math. (2)}, 117(1):77--108, 1983.

\bibitem{Cekic2017a}
M.~Ceki\'{c}.
\newblock Calder\'{o}n problem for connections.
\newblock {\em Comm. Partial Differential Equations}, 42(11):1781--1836, 2017.

\bibitem{Cekic2020}
M.~Ceki\'{c}.
\newblock Calder\'{o}n problem for {Y}ang-{M}ills connections.
\newblock {\em J. Spectr. Theory}, 10(2):463--513, 2020.

\bibitem{CLOP}
X.~Chen, M.~Lassas, L.~Oksanen, and G.~Paternain.
\newblock Detection of {H}ermitian connections in wave equations with cubic
  non-linearity.
\newblock {\em J. Eur. Math. Soc. (JEMS)}, 24(7):2191--2232, 2022.

\bibitem{CLOP2}
X.~Chen, M.~Lassas, L.~Oksanen, and G.~P. Paternain.
\newblock Inverse problem for the {Y}ang-{M}ills equations.
\newblock {\em Comm. Math. Phys.}, 384(2):1187--1225, 2021.

\bibitem{Duistermaat-Hormander-FIO2}
J.~J. Duistermaat and L.~H\"{o}rmander.
\newblock Fourier integral operators. {II}.
\newblock {\em Acta Math.}, 128(3-4):183--269, 1972.

\bibitem{Feizmohammadi2020}
A.~Feizmohammadi and L.~Oksanen.
\newblock An inverse problem for a semi-linear elliptic equation in
  {R}iemannian geometries.
\newblock {\em J. Differential Equations}, 269(6):4683--4719, 2020.

\bibitem{Feizmohammadi2022}
A.~Feizmohammadi and L.~Oksanen.
\newblock Recovery of zeroth order coefficients in non-linear wave equations.
\newblock {\em J. Inst. Math. Jussieu}, 21(2):367--393, 2022.

\bibitem{Greenleaf-Uhlmann-CMP-1993}
A.~Greenleaf and G.~Uhlmann.
\newblock Recovering singularities of a potential from singularities of
  scattering data.
\newblock {\em Comm. Math. Phys.}, 157(3):549--572, 1993.

\bibitem{Ha}
M.~J.~D. Hamilton.
\newblock {\em Mathematical gauge theory}.
\newblock Universitext. Springer, Cham, 2017.
\newblock With applications to the standard model of particle physics.

\bibitem{Hormander-Vol4}
L.~H\"{o}rmander.
\newblock {\em The analysis of linear partial differential operators. {IV}},
  volume 275 of {\em Grundlehren der Mathematischen Wissenschaften}.
\newblock Springer-Verlag, Berlin, 1985.

\bibitem{H1}
L.~H\"{o}rmander.
\newblock {\em The analysis of linear partial differential operators. {I}}.
\newblock Springer Study Edition. Springer-Verlag, Berlin, second edition,
  1990.

\bibitem{KLOU}
Y.~Kurylev, M.~Lassas, L.~Oksanen, and G.~Uhlmann.
\newblock Inverse problem for {Einstein}-scalar field equations.
\newblock {\em Duke Math. J.}, 171(16):3215--3282, 2022.

\bibitem{KLU}
Y.~Kurylev, M.~Lassas, and G.~Uhlmann.
\newblock Inverse problems for {L}orentzian manifolds and non-linear hyperbolic
  equations.
\newblock {\em Invent. Math.}, 212(3):781--857, 2018.

\bibitem{KOP}
Y.~Kurylev, L.~Oksanen, and G.~P. Paternain.
\newblock Inverse problems for the connection {L}aplacian.
\newblock {\em J. Differential Geom.}, 110(3):457--494, 2018.

\bibitem{Matti}
M.~Lassas.
\newblock Inverse problems for linear and non-linear hyperbolic equations.
\newblock In {\em Proceedings of the {I}nternational {C}ongress of
  {M}athematicians---{R}io de {J}aneiro 2018. {V}ol. {IV}. {I}nvited lectures},
  pages 3751--3771. World Sci. Publ., Hackensack, NJ, 2018.

\bibitem{Lassas2021}
M.~Lassas, T.~Liimatainen, Y.-H. Lin, and M.~Salo.
\newblock Inverse problems for elliptic equations with power type
  nonlinearities.
\newblock {\em J. Math. Pures Appl. (9)}, 145:44--82, 2021.

\bibitem{Lassas2017a}
M.~Lassas, G.~Uhlmann, and Y.~Wang.
\newblock Determination of vacuum space-times from the {E}instein-{M}axwell
  equations.
\newblock {\em Preprint arXiv:1703.10704}.

\bibitem{LUW}
M.~Lassas, G.~Uhlmann, and Y.~Wang.
\newblock Inverse problems for semilinear wave equations on {L}orentzian
  manifolds.
\newblock {\em Comm. Math. Phys.}, 360(2):555--609, 2018.

\bibitem{Melrose-Uhlmann-CPAM1979}
R.~B. Melrose and G.~A. Uhlmann.
\newblock Lagrangian intersection and the {C}auchy problem.
\newblock {\em Comm. Pure Appl. Math.}, 32(4):483--519, 1979.

\bibitem{Oksanen2020}
L.~Oksanen, M.~Salo, P.~Stefanov, and G.~Uhlmann.
\newblock Inverse problems for real principal type operators.
\newblock {\em Am. J. Math.}, 146(1):161--240, 2024.

\bibitem{Tataru}
D.~Tataru.
\newblock Unique continuation for solutions to {PDE}'s; between
  {H}\"{o}rmander's theorem and {H}olmgren's theorem.
\newblock {\em Comm. Partial Differential Equations}, 20(5-6):855--884, 1995.

\bibitem{Tetlow2022}
A.~Tetlow.
\newblock Recovery of a time-dependent {H}ermitian connection and potential
  appearing in the dynamic {S}chr\"{o}dinger equation.
\newblock {\em SIAM J. Math. Anal.}, 54(2):1347--1369, 2022.

\bibitem{Tong2017}
D.~Tong.
\newblock Line operators in the {S}tandard {M}odel.
\newblock {\em J. High Energy Phys.}, (7):104, front matter+14, 2017.

\bibitem{Uhlmann2018}
G.~Uhlmann and Y.~Wang.
\newblock Determination of space-time structures from gravitational
  perturbations.
\newblock {\em Comm. Pure Appl. Math.}, 73(6):1315--1367, 2020.

\end{thebibliography}

\ifoptionfinal{}{
}

\end{document}